\documentclass[11pt]{amsart}

\usepackage{amsmath}
\usepackage[dvips]{epsfig}
\usepackage[english]{babel}
\usepackage[utf8x]{inputenc}
\usepackage[T1]{fontenc}

\usepackage{amssymb,amsmath,amsthm,amsfonts, color,graphicx}
\usepackage{graphicx}
\usepackage[colorinlistoftodos]{todonotes}
\usepackage[colorlinks=true, allcolors=blue]{hyperref}

\usepackage{amsrefs}

\usepackage{tikz} %This package offers the ability to draw pictures
\usepackage{tikz-cd}
\usetikzlibrary{positioning}
\tikzset{>=stealth}

\usepackage{caption} %This package allows me to rename my table
\usepackage{xcolor} %This package allows me to add colors to the equations
\usepackage{mathtools} %This package allows me to write texts under arrows

\usepackage{mathrsfs}  %This allows me to use \mathscr

\usepackage{wasysym} %This allows me to use musical symbols

\usepackage{tikz}

\newtheorem{mydef-no-caption}{Definition}
	{\begin{mydef-no-caption}{\ifnotmtarg{#1}{\textnormal{(\textbf{#1})}~}}}%
	{\end{mydef-no-caption}}

\theoremstyle{plain}
\newtheorem{thm}{Theorem}[section]
\newtheorem{prop}[thm]{Proposition}
\newtheorem{lem}[thm]{Lemma}
\newtheorem{cor}[thm]{Corollary}

\theoremstyle{definition}
\newtheorem{rmk}[thm]{Remark}
\newtheorem{defn}[thm]{Definition}

\numberwithin{equation}{section}
\renewcommand{\bar}{\overline}
\newcommand{\N}{{\mathbb N}}

\newcommand{\R}{{\mathbb R}}

\newcommand{\lan}{\langle}
\newcommand{\ran}{\rangle}
\newcommand{\pa}{\partial}

\newcommand{\bsls}{\backslash}
\newcommand{\tr}{\mathrm{tr}}
\newcommand{\Div}{\mathrm{div}}
\newcommand{\D}{\nabla}

\newcommand{\bigO}{\mathcal{O}}

\newcommand{\area}{\mathrm{Area}}
\newcommand{\Ric}{\mathrm{Ric}}
\newcommand{\dist}{\mathrm{dist}}
\newcommand{\gr}{\mathrm{Graph}}
\newcommand{\fgr}{\mathfrak{Graph}}
\newcommand{\cA}{\mathcal{A}}
\newcommand{\cC}{\mathcal{C}}
\newcommand{\fM}{\mathfrak{M}}
\newcommand{\fg}{\mathfrak{g}}
\newcommand{\fD}{\mathfrak{D}}
\newcommand{\rH}{\mathrm{H}}
\newcommand{\rK}{\mathrm{K}}
\newcommand{\cN}{\mathcal{N}}
\newcommand{\cL}{\mathcal{L}}
\newcommand{\cF}{\mathcal{F}}
\newcommand{\cG}{\mathcal{G}}

\newcommand{\cP}{\mathcal{P}}
\newcommand{\rR}{\mathrm{R}}
\newcommand{\cS}{\mathcal{S}}
\newcommand{\cT}{\mathcal{T}}

\newcommand{\al}{\alpha}

\newcommand{\ga}{\gamma}
\newcommand{\Ga}{\Gamma}
\newcommand{\la}{\lambda}
\newcommand{\La}{\Lambda}
\newcommand{\si}{\sigma}
\newcommand{\Si}{\Sigma}
\newcommand{\De}{\Delta}
\newcommand{\de}{\delta}
\newcommand{\ka}{\kappa}
\newcommand{\vare}{\varepsilon}
\newcommand{\om}{\omega}
\newcommand{\Om}{\Omega}
\renewcommand{\th}{\theta}
\newcommand{\Th}{\Theta}
\newcommand{\Up}{\Upsilon}

\DeclarePairedDelimiter\abs{\lvert}{\rvert}
\DeclarePairedDelimiter\norm{\lVert}{\rVert}
\title{On blowup of regularized solutions to Jang equation and constant expansion surfaces}
\author{Kai-Wei Zhao}
\date{}

\begin{document}
\begin{abstract}
In this paper, we analyze the blowup behavior of regularized solutions to Jang equation inside apparent horizons. This extends the analyses outside apparent horizons done by Schoen-Yau in \cite{SY2}. We will take two natural geometric treatments to blowup sequences: dilation and translation. First, we show that the limits of properly translated solutions are constant expansion surfaces. Second, we characterize the limits of properly rescaled solutions. Third, we discuss the structure of blowup regions enclosed by apparent horizons. Lastly, we elaborate on a special case of low-speed blowup. 
\end{abstract}

\maketitle

\section{Introduction}
\subsection{Initial data set and apparent horizon}
An \emph{initial data set} is a triple $(M,g,k)$ where $M$ is a complete smooth 3-manifold without boundary, equipped with a Riemannian metric $g$ and a symmetric $(0,2)$-tensor $k$. In general relativity, initial data sets arises as embedded spacelike hypersurfaces of time-orientable Lorentzian manifolds $(\fM, \fg)$, referred to as spacetime, with induced metric $g$ and second fundamental form $k$. The constraint equations defines \emph{local mass density} $\mu$ and the \emph{local current desity} $J$ of an initial data set $(M,g,k)$ as 
\begin{align*}
    \mu: = \frac{1}{2}\big(R_g - \abs{k}^2_g + \tr_g(k)^2\big),\quad J:= \Div\big(k - \tr_g(k) g\big)
\end{align*}
where $R_g$ denotes the scalar curvature of $(M,g)$. An initial data set $(M,g,k)$ satisfies \emph{dominant energy condition} if $\mu\geq \abs{J}_g$ holds on $M$.

An initial data set $(M,g,k)$ is \emph{asymptotically flat} if there is a compact subset $K\subset M$ such that $M\bsls K$ consists of finite number of connected components $M_1, \ldots, M_p$ each of which is diffeomorphic to $\R^3\bsls \bar{B}$ for some closed ball $\bar{B}$ in $\R^3$. In this paper, we will also require several drop-off rates on such coordinate charts at each infinity end: 
\begin{align}\label{condition: asymp flat}
    g_{ij}-\de_{ij} = \bigO^2(\abs{x}^{-1}), \quad R_g = \bigO^1(\abs{x}^{-4}), \quad k_{ij}\in \bigO^2(\abs{x}^{-2}), \quad \sum_{i = 1}^3 k_{ii} = \bigO(\abs{x}^{-3}).
\end{align}
Here, a function $f = \bigO^k(|x|^{-p})$ means that $\sup_{M\backslash K} \sum_{i = 0}^k |x|^{p + i} |\pa^i f| < \infty$.

Suppose $(M,g,k)$ is an initial data set in spacetime $(\fM, \fg)$, then there is a future pointing normal $\eta$ of $M$ in $\fM$ such that $\fg(\eta, \eta) = -1$, and at any $p\in M$, $k(X,Y) = \fg(\fD_X \eta, Y)$ for all $X,Y\in T_pM$. Here, $\fD$ is the Levi-Civita connection of the spacetime $(\fM, \fg)$. Let $\Om\subseteq M$ be an open region and $\Si = \pa\Om$ be a smooth embedded two-sided surface and let $\nu$ be the unit normal vector field on $\Si$ pointing out of $\Om$ in $M$. Let $h$ denote the second fundamental form of $\Si$ in $M$ with respect to $\nu$ so that $h_p(X,Y): = g(\D_X \nu, Y)$ for all $X,Y\in T_p\Si$ for all $p\in\Si$ and let $\rH$ denote the mean curvature with respect to $\nu$. Now we think of $\Si$ as a spacelike 2-surface embedded in $\fM$. There are two independent canonical null normal vector fields $l^+:= \eta+\nu$ and $l^-:= \eta-\nu$ on $\Si$ in $\fM$. Then we can define the null second fundamental form $\mathfrak{h}^\pm$ of $\Si$ in $\fM$ along $l^{\pm}$ by $\mathfrak{h}^{\pm}_p(X,Y):=\fg(\fD_X l^{\pm}, Y) = (\pm h+k)(X,Y)$ for all $X,Y\in T_p\Si$ for all $p\in\Si$. We call respectively $\th^\pm := \tr_\Si \mathfrak{h}^{\pm} = \pm \mathrm{H}_\Si+\tr_\Si k$ the outer (+) or inner (-) \emph{null expansion} of $\Si$. If $\Si$ is a 2-surface with $\th^+_\Si = 0$, then $\Si$ is called a \emph{marginally outer trapped surface} (MOTS). If $\th^-_\Si=0$, then $\Si$ is called a \emph{marginally inner trapped surface} (MITS). We call $\Si$ an \emph{apparent horizon} if it is either a MOTS or MITS. Throughout this paper, we will always take a unconventional definition of expansion $\th_\Si := \rH_\Si - \rK_\Si$ and specify the choice of unit normal vector for the purpose of simplicity.

\subsection{Jang equation and regularized equations}
In 1981, R. Schoen and S.T. Yau \cite{SY2} proved general positive energy theorem by reducing the problem to the time-symmetric case which they had proved in 1979 \cite{SY1}. They noticed that the positive energy theorem in the time-symmetric case characterizes Euclidean space $\mathbb{R}^n$ as the space of zero energy. On the other hand the spacetime positive energy theorem would characterizes an arbitrary spacelike hypersurface in Minkowski space. Therefore they considered the equation introduced by P.S. Jang \cite{J} which consists of such characterization.

Given a Riemannian manifold $(M^3, g)$, it is a spacelike hypersurface in Minkowski space $\mathbb{R}^{1,3}$ if and only if it is the graph of a function $f(x)$ defined on $M$ with metric
\begin{align*}
    g_{ij} = \delta_{ij} - \pa_i f \pa_j f
\end{align*}
with $|Df|^2 < 1$. The equation is equivalent to 
\begin{align*}
    \delta_{ij} = g_{ij} + \D_i f \D_j f
\end{align*}
where $\D$ is the connection on $M$. It it convenient to think of $\R^3$ as a graph of $f$ in Riemannian manifold $M\times \R$ with product metric $g+ dt^2 $. Jang also observed that the second fundamental form $k$ of $M$ in Minkowski space agrees with the one of the graph of $f$ in $M\times \R$
\begin{align*}
    k_{ij} = \frac{\D_i \D_j f}{\sqrt{1+|\D f|^2}}
\end{align*}
with respect to the downward unit normal vector. Consequently, the triple $(M,g,k)$ is embeddable in Minkowski space as a spacelike hypersurface with metric $g$ and second fundamental form $k$ if and only if there is a function $f$ defined on $M$ whose graph has flat metric and second fundamental form $k$. Nevertheless, the system of equations is overdetermined and is usually unsovable. The Jang equation is a trace equation of the system:
\begin{align*}
    \sum_{i,j} \big(g^{ij} - \frac{f^i f^j}{\sqrt{1+|\D f|^2}}\big)\big(\frac{\D_i \D_j f}{\sqrt{1+|\D f|^2}}-k_{ij}\big) = 0
\end{align*}
where $f^i= g^{ij} f_j $ and the first factor is the inverse of induced metric on the graph of $f$ in $(M\times \R, dt^2+g)$. In the equation, the first term
\begin{align*}
    \sum_{i,j} \big(g^{ij} - \frac{f^i f^j}{\sqrt{1+|\D f|^2}}\big)\frac{\D_i \D_j f}{\sqrt{1+\abs{\D f}^2}} = \Div_M (\frac{\D f}{\sqrt{1+|\D f|^2}}) =: \rH(f)
\end{align*}
is the mean curvature of graph($f$) with respect to downward unit normal and the second term 
\begin{align*}
    \sum_{i,j} \big(g^{ij} - \frac{f^if^j}{\sqrt{1+|\D f|^2}}\big)k_{ij} = \tr_{\gr(f)} k =: \rK(f).
\end{align*}
is the trace on the tangent space of graph($f$) of the tensor $k$ extended trivially in the vertical direction to $M\times \R$, i.e. $k(\pa_t, \cdot) = 0$. Then Jang equation is a marginally trapped surface equation
\begin{align*}
    \mathrm{H}(f) - \mathrm{K}(f) = 0
\end{align*}
in the new initial data set $(M\times \R, dt^2+g, k)$.

To study the existence and regularity properties of Jang equation, Schoen-Yau in section 4 of \cite{SY2} (also cf. \cite{E2}) introduced a family of regularized equations and they showed (Lemma 3 in \cite{SY2}) that for every $s>0$ there exists a unique smooth solution $f_s$ of regularized equation 
\begin{align}\label{eq: regularized eq}
    (g^{ij}- \frac{f_s^i f_s^j}{1+|\D f_s|^2})\Big(\frac{\D_i\D_j f_s}{\sqrt{1+|\D f_s|^2}} - k_{ij}\Big) = sf_s.
\end{align}
satisfying $\lim_{x\to \infty} f_s(x) = 0$ at each infinity end. By the uniform a priori estimate based on stability argument in \cite{SY2}, a sequence of regularized solutions converges smoothly to a solution to Jang equation, while the regularized solutions could blow up to $\pm\infty$ in some bounded region enclosed by apparent horizons. Schoen-Yau analyzed the blowup behavior of the solution to Jang equation on approach to apparent horizons from outside, but what happens inside the apparent horizon remains unknown. 

In this present paper, we will take two natural geometric treatments to blowup sequences: \emph{dilation} and \emph{translation}. The key observation leading to the blowup analyses in this paper is that the capillary terms in the regularized equations can be regarded as properly rescaled solutions. In particular, a subsequence of capillary rescaled solutions converges to a Lipschitz function, called a \emph{capillary blowdown limit}. We can think of the capillary blowdown limit as a compressed profile of the limit of a blowup sequence. It turns out that capillary blowdown limit play an important role in the analysis of the limit of translated regularized solutions. Roughly speaking, the capillary blowdown limit portraits a crude but global picture about blowup behavior; the limit of translated solutions elaborates on rather accurate but local information.

In Section 2, we start by defining several frequently used notations in this paper. Next, we recall some important properties of regularized solutions in \cite{SY2}, including a priori estimates and Harnack-type inequality for a more general equation. Next, we introduce Fermi coordinates, which is technically useful, and the linearized operator of expansion. Then we extend the definition of stability of marginally trapped surfaces in \cite{AMS} to constant expansion surface.

In Section 3, we first show the existence of capillary blowdown limit for any blowup sequence and prove a basic lemma, which states that any limit of translated regularized solutions must stay in the cylinder over a level-set of capillary blowdown limit. Next, we prove several existence results of limits of translated regularized solutions. In each corresponding scenario, we prove the characterization theorem, which states that any limit of translated regularized solutions is a either graphical or cylindrical \emph{constant expansion surface} in $M\times \R$, on which the constant expansion is determined by the capillary blowdown limit. Lastly, we show that all these constant expansion surfaces arise as the limits of translated regularized solutions are stable.

In Section 4, we find that capillary blowdown limits are weak viscosity solutions to a singular elliptic geometric equation, which is related to foliations of constant expansion surfaces. Next, we prove the a priori estimates for foliations of stable constant expansion surfaces. We then construct a smooth (local) solution to the aforementioned geometric equation around any stable constant expansion surface. A comparison theorem between capillary blowdown limit and the constructed smooth solution is given at the end of Section 4.

In Section 5, the main theorem is the structure theorem of blowup regions of regularized solutions. To get started, we show two properties of thin maximal domains: The first property states that any thin maximal domain is an annular region; The second property states that if the maximal domain is thin enough, then its closure contains part of a closed smooth marginally stable constant expansion surface. Next, we prove the structure theorem. Briefly speaking, the idea is to apply the existence and characterization theorems in Section 3 at points in a fixed countable dense subset of blowup regions. Using diagonal argument gives a subsequence of regularized solutions such that its properly translated sequences around every points in the blowup regions converge to a collection of disjoint graphical or cylindrical constant expansion surfaces in $M\times \R$. This collection of disjoint constant expansion surfaces leads to a partition of the blowup regions as a disjoint union of maximal domains (of graphical solutions to constant expansion equations) and foliations of constant expansion surfaces. Furthermore, these maximal domains and sheets of foliations of constant expansion surfaces make up the level-sets of capillary blowdown limit. Therefore, the structure of capillary blowdown limit is also understood. At the end of the section, we show a volume estimate and a boundary area estimate for blowup regions as immediate applications of the structure theorem.

In Section 6, we elaborate on properties of trivial capillary blowdown limit which is identically zero everywhere. This is a special case of slow blowup behavior. We first show that trivial capillary blowdown limit is rigid in the sense that if one sequence of blowup sequence of regularized solutions has trivial capillary blowdown limit, then all other sequences do. Next, we prove that trivial capillary blowdown limit gives a topological restriction to blowup regions assuming dominant energy condition. In the model case where the blowup region is one maximal domain of a solution to Jang equation, we use the stability inequality and a gluing lemma to construct a closed orientable smooth 3-manifold of positive Yamabe type. It is known that a closed orientable smooth 3-manifold of positive Yamabe type is a connected sum of finite numbers of space forms and $S^2 \times S^1$'s. We conclude the same result for general case by applying the structure theorem of blowup region and property of thin maximal domains in Section 5.

\subsection*{Acknowledgments}
This paper forms part of my thesis. I am extremely grateful to my advisor Rick Schoen for introducing this problem to me, and for his constant support and encouragement, especially during the pandemic. I would like to sincerely thank Michael Eichmair for clarifying the proof of stability of MOTS and for many helpful suggestions. Thanks very much to Long-Sin Li, Chao-Ming Lin, Hongyi Sheng, and Tin-Yau Tsang for great discussions. 

\section{Preliminaries}
\subsection{Notations}
Let $u$ be a function defined on $M$ and let $C\in \R$ be a number. Let
\begin{equation*}
	E_C^+(u) := \{x\in M: u(x) > C\}
\end{equation*} 
denote the super-level set of $u$, and let
\begin{equation*}
	E_C^-(u) := \{x\in M: u(x) < C\}
\end{equation*}
denote the sub-level set of $u$.

\subsection{Properties of regularized solutions}
\begin{prop}[\cite{SY2}, Proposition 1 and 2] \label{Prop: LocalEst}
Let $F\in C^1(M)$ and $\mu_1, \mu_2$ be constants so that 
\begin{align}
    \sup_M |F| \leq \mu_1,\quad \sup_M |\D F|\leq \mu_2.
\end{align}
Suppose $f$ is a $C^2$ solution to 
\begin{align}
    \mathrm{H}(f) - \mathrm{K}(f) = F(x).
\end{align}
Then
\begin{enumerate}
    \item[(1)] There exists $c_1 = c_1(M,g,k,\mu_1,\mu_2)$ such that the norm of the second fundamental form 
\begin{equation}
	|h|^2 \leq c_1.
\end{equation} 
    \item[(2)] There is $\rho = \rho(M,g,k,\mu_1,\mu_2)>0$ such that for every $X_0\in \gr(f)$ and $(y^1,y^2,y^3, y^4)$ normal coordinates in $M\times \R$ on which $T_{X_0} \gr(f)$ is the $y^1y^2y^3$-space, the local defining function $w(y)$ for $\gr(f)$ is defined on $\{y=(y^1,y^2,y^3): |y|\leq \rho\}$ with 
    \begin{align*}
        \gr(f)\cap B^4 (X_0;\frac{\rho}{2}) \subseteq \{(y, w(y)): |y|\le \rho\}
    \end{align*} and satisfies for any $\al\in (0,1)$ there is a constant $c_2 = c_2(M,g,k,\mu_1,\mu_2,\al)>0$ such that
    \begin{align*}
        \norm{w}_{3,\al;\{y:|y|\leq \rho\}} \leq c_2.
    \end{align*}
    Here, $B^4(X_0, r)$ denotes the geodesic ball in $(M\times \R, g + dt^2)$ and $\|w\|_{3,\al; \{y:|y|\leq \rho\}}$ denotes the $C^{3,\al}$-Holder norm on the Euclidean ball $\{y:|y|\leq \rho\}$ in the tangent space.
    \item[(3)] There are constants $c_3,c_4$ depending on $M,g,k,\mu_1,\mu_2$ such that the following Harnack-type inequalities hold
    \begin{align*}
        &\sup_{\gr(f)\cap B^4 (x_0;\frac{\rho}{2})} \lan \nu, -\pa_t \ran \leq c_3  \inf_{\gr(f)\cap B^4(X_0;\frac{\rho}{2})} \lan \nu, -\pa_t \ran\\ 
        &\sup_{\gr(f)\cap B^4 (x_0;\frac{\rho}{2})} |\bar{\D} \log \lan \nu, -\pa_t \ran| \leq c_4.
    \end{align*}
    Here, $\nu$ is the downward pointing normal of $\gr(f)$ in $M\times \R$ and $\bar{\D}$ denotes the Levi-Civita connection on $\gr(f)$.
\end{enumerate}
\end{prop}

\begin{rmk}
In \cite{SY2}, they also assumed the bound for $|\D^2F|$, but this condition was not used in their proof.
\end{rmk}

To apply the local estimates in Proposition \ref{Prop: LocalEst} to the regularized equations (\ref{eq: regularized eq}), Schoen and Yau proved by maximum principle argument that there are constants $\mu_1 = \max_M |\tr_g k|$ and $\mu_2 = \mu_2(|\mathrm{Ric}|_{C^0(M)}, |k|_{C^1(M)})$ such that 
\begin{align}
	|sf_s|\leq \mu_1 \quad \mbox{and} \quad   |s \D f_s|\leq \mu_2 \quad \mbox{in $M$}.
\end{align}
Using the a priori estimates Proposition \ref{Prop: LocalEst}, they showed the existence of convergent sequence $f_{s_j}$ and the regularity of the limit function $f_0$ in the following proposition. Moreover, since the capillary terms in (\ref{eq: regularized eq} converge to zero in the region where $f_0$ is finite, $f_0$ is a solution to Jang equation.
\begin{prop}[cf. \cite{SY2}, \cite{E2}]\label{prop: SY Jang equation}
There exists a positive sequence $s_j\to 0$ and disjoint open sets $\Omega_+, \Omega_-, \Omega_0$ with the following properties:
\begin{enumerate}
    \item[(1)] $f_{s_j}$ diverges to $\pm \infty$ on $\Omega_\pm$ respectively and $f_{s_j}$ converges to a smooth function $f_0$ on $\Omega_0$ which satisfies Jang equation $\mathrm{H}(f_0) - \mathrm{K}(f_0) = 0$ and drops off at the rate $f_0\in \bigO^{3}(\abs{x}^{-1/2})$ at each infinity of $M$.
    \item[(2)] The sets $\Omega_+$ and $\Omega_-$ have compact closures and $M = \bar{\Omega}_+\cup \bar{\Omega}_- \cup \bar{\Omega}_0$. Each connected component $\Sigma_\pm$ of $\pa\Omega_\pm$ is a closed properly embedded smooth apparent horizon in $M$ satisfying $\mathrm{H}_{\Sigma_\pm} \pm \tr_{\Sigma_\pm} k = 0$ where $\mathrm{H}_{\Sigma_\pm}$ is computed with respect to the unit normal on $\pa\Omega_\pm$ pointing out of $\Om_\pm$. No two connected components of $\Om_+$ (respectively $\Om_-$) can share a common boundary. Moreover, as $a\to \pm \infty$ the hypersurfaces $\gr(f_0 - a, \Omega_0)$ converge to the cylinder $(\pa \Omega_\pm \cap \pa \Omega_0)\times \R$ uniformly in the sense of $C^{2,\al}_{loc}$.
    \item[(3)] $\gr(f_{s_j})$ converges smoothly to a hypersurface $S$ in $M\times \R$. Each component of $S$ is either a component of $\gr(f_0,\Omega_0)$ or a cylinder $\Si\times\R$ over a component $\Sigma$ of $\pa\Omega_+\cap \pa \Om_-$. Any two components of $S$ are separated by a positive distance.
\end{enumerate}
\end{prop}

\subsection{Fermi coordinates}
Let $(M^{n+1},g)$ be a $(n+1)$-dimensional Riemannian manifold. Let $\Si^n\subset M$ be a smooth embedded two-sided $n$-dimensional surface assigned with unit normal vector field $\nu$.  Let $U', U$ be open subset of $M$, let $y^1, \ldots, y^n$ be a coordinate system on $\Si \cap U'$ and let $\de>0$ small such that the map $\Up: \Si\cap U'\times (-\de, \de) \rightarrow U$ given by 
\begin{align*}
    \Up(y, \si) = \exp_y(\si \nu (y))
\end{align*} is bijective. This means that for any point $p\in U$, there exists a unique $y\in \Si$ and $\si\in (-\de, \de)$ such that $p = \Up(y, \si)$. We then introduce the Fermi coordinates $y^1, \ldots, y^n, \si$ on $U$ through the map $\Up$. We denote basis vectors by $\pa_i = \frac{\pa}{\pa y^i}$ for all $i$ and $\pa_\si = \frac{\pa}{\pa \si}$. By properties of exponential map, we have $\lan \pa_i, \pa_\si\ran(p) = 0$ for $1\leq i \leq n$ and $\D_{\pa_\si} \pa_\si(p) = 0$ for all $p\in U$. In Fermi coordinates the metric $g$ in $U$ can be written as
\begin{align*}
    \sum_{i,j = 1}^n g_{ij}(y, \si) dy^i dy^j + d\si^2 = g(y, \si).
\end{align*}
where $g_{ij} = \lan \pa_i, \pa _j\ran$. If we define the surface $\Si_\si = \{\si \equiv \mathrm{const}\}$ in Fermi coordinates then
\begin{align*}
	&g\big|_{\Si_\si} = \sum_{i,j = 1}^n g_{ij}(y, \si) dy^i dy^j,\\
	&\pa_\si g_{ij}(y, \si) = 2 h_{ij}(y, \si).
\end{align*}
where $h_{ij}(y, \si) := \lan \D_{\pa_i} \pa_\si, \pa_j \ran(y, \si)$ is the second fundamental form of $\Si_\si$ with respect to $\pa_\si$. When $\si\equiv 0$, the Fermi coordinate system nicely captures the geometry of $\Si$ in $M$. We denote half geodesic tubular neighborhood around $\Si$ on the $\pm\nu$-side with thickness $\de$ respectively by
\begin{align*}
    \cN_{\pm,\de}(\Si, \nu) := \{\Up(y, \pm\si): x\in \Si, 0 \leq \si < \de)\},
\end{align*}
and the full tubular neighborhood around $\Si$ with thickness $2\de$ by
\begin{align*}
    \cN_\de(\Si) := \{y\in M: \dist(x, \Si)< \de\}.
\end{align*}

Sometimes we will analyze the properties of constant expansion surfaces near another. It would be useful to consider graphs in Fermi coordinates. Let $w\in C^\infty(\Si)$ with $|w|< \de$, let $\fgr(w) = \{\Up(y, w(y)): y\in \Si\}$ denote the graph of $w$ in Fermi coordinates.

\subsection{Linearization of the expansion}
Let $\Si^n\subset M^{n+1}$ be a smooth embedded two-sided hypersurface in an initial data set $(M^n, g, k)$ and let $\nu$ be the normal vector field assigned to $\Si$. Let $\Phi_\tau$ be a smooth one-parameter family of differeomorphisms of $M$ for $\tau\in (-\vare, \vare)$ so that $\Phi_0$ is the identity map. Suppose that $\frac{d}{d\tau}|_{\tau=0} \Phi_\tau = X + \varphi \nu$ where $X$ is a tangential vector field, $\varphi$ is a smooth function on $\Si$. Denote $\Si_\tau : = \Phi_\tau(\Si)$. We have the following variation formulas (c.f. \cite{M1} Lemma 5.1 and \cite{AEM} section 2.2)
\begin{align}
    \frac{d}{d\tau}\Big|_{\tau=0} \,\Phi^*_\tau \mathrm{H}_{\Si_\tau} &= \lan\D^\Si \,\mathrm{H}_\Si, X\ran - \De^\Si \varphi - \big(|h|^2_\Si + \Ric(\nu,\nu)\big)\varphi\\
    \frac{d}{d\tau}\Big|_{\tau=0} \,\Phi^*_\tau \tr_{\Si_\tau}(k) &= \lan \D^\Si \,\tr_\Si(k), X\ran + 2k(\nu, \D^\Si \varphi) + \D_{\nu} \big(\tr_{M}(k)\big)\varphi - (\D_{\nu} k)(\nu,\nu)\varphi
\end{align}
where $\D^\Si$ and  $\De^\Si$ denote respectively the gradient operator and non-positive Laplacian operator on $\Si$ equipped with induced metric, $|h|_\Si^2$ denotes the square norm of the second fundamental form of $\Si$ in $M$ with respect to $\nu$, and $\Ric$ and $D$ denote ambient Ricci curvature and  Levi-Civita connection in $M$. Now let $\xi := \big(k(\nu,\cdot)^\sharp\big)^\top \in \Gamma(\mathrm{T}\Si)$, we have
\begin{align*}
    (D_{\nu}\, k)(\nu,\nu) = -\mathrm{H}_\Si\, k(\nu,\nu) + \lan h, k\ran_\Si + (\Div_{M}\, (k))(\nu)- \Div_\Si\, (\xi).
\end{align*}
Using the Gauss equation and the definition of local density mass $\mu$ of $(M, g, k)$, we can compute
\begin{align*}
    \Ric(\nu,\nu) = \mu +\frac{1}{2}\Big(-\rR_\Si + |k|_{g} - \big(\tr_{g} \, k\big)^2 - |h|^2_\Si + \mathrm{H}^2_\Si\Big)
\end{align*}
and using definition of local current density $J$ of $(M, g, k)$ we have
\begin{align*}
    \big(\Div_{M}\,(k)\big)(\nu) = J(\nu) + D_{\nu} (\tr_\Si \,k).
\end{align*}
Define $\th_{\Si_\tau}:= \rH_{\Si_\tau} - \rK_{\Si_\tau}$ to be the expansion of $\Si_\tau$. Combining all above identities, we obtain
\begin{align*}
    \frac{d}{d\tau}\Big|_{\tau=0} \,\Phi^*_\tau \,\th_{\Si_\tau} &= \lan D^\Si \,\th_\Si , X \ran - \De^\Si \varphi - 2\lan\xi, \D^\Si \varphi \ran\\
    &\quad + \bigg(\frac{1}{2}\rR_\Si - \frac{1}{2}|h - k|^2_\Si - \mu + J(\nu) - \Div_\Si \xi - |\xi|^2 - \frac{1}{2}\th_\Si\,\big( 2\tr_{M}\,(k) + \th_\Si \big)\bigg)\varphi.
\end{align*}
If $\Phi$ is a normal deformation of the constant expansion hypersurface $\Si$ which means that $X=0$ and $\th_\Si \equiv \Th$ is constant, then we can define the \emph{linearized operator of expansion} by
\begin{align}\label{op1}
    \mathcal{L}_\Si \varphi :=\frac{d}{d\tau}\Big|_{\tau=0} \,\Phi^*_\tau \,\th_{\Si_\tau} =  - \De^\Si \varphi - 2\lan \xi, \D^\Si \varphi \ran 
    + \bigg(\cP - \Div_\Si \xi - |\xi|^2 - \frac{1}{2}\Th\,\big(2\tr_{M}\,(k) + \Th\big)\bigg)\varphi
\end{align}
where $\cP = \frac{1}{2}\rR_\Si - \frac{1}{2}|h - k|^2_\Si - \mu + J(\nu)$. If $\varphi>0$, we have a simpler expression
\begin{align}\label{op2}
    \varphi^{-1}\cL_{\Si} \varphi &= -\Div_{\Si}(\xi + \D^\Si \log \varphi) - |\xi + \D^\Si \log \varphi|_\Si^2 + \cP - \frac{1}{2}\Th\,\big( 2\tr_{M}\,(k) + \Th\big).
\end{align}

The linear operator $\cL_\Si$ is not self-adjoint due to the first-order derivative contributed by $k$. Thus, apparent horizons do not arise as critical points of a standard variational problem in terms of initial data set $(M,g,k)$. When $\Si$ is closed, the Krein-Rutman theorem in general elliptic operator theory implies that the \emph{principal eigenvalue} $\la_1 = \la_1(\cL_\Si)$ is real and that there is a smooth positive eigenfunction $\beta$ defined on $\Si$ satisfying $\cL_\Si \beta = \la_1 \beta$. Recall that the principal eigenvalue is the eigenvalue of $\cL_\Si$ having the minimal real part. Moreover, $\la_1$ is \emph{simple}. Namely, the dimension of the eigenspace corresponding to $\la_1$ is one. More discussion can be found in \cite{AMS} section 4. We say that a constant expansion surface $\Si$ is \emph{stable} if the principal eigenvalue $\la_1$ of $\cL_\Si$ is nonnegative.
\begin{prop}[cf. \cite{AM}, also \cite{AEM} and Proposition \ref{prop: stability of CES} for constant expansion surfaces in the present paper for simplified proof] \label{prop: stability for MOTS}
The closed smooth apparent horizons appear as components of $\pa \Om_\pm , \pa \Om_0$ in Proposition \ref{prop: SY Jang equation} are stable. 
\end{prop}

\section{Limits of regularized solutions}
\subsection{Capillary blowdown limit}
Recall that for every $s\in (0,1]$ there exists a unique smooth regularized solution $f_s$ such that
\begin{align*}
    (g^{ij}- \frac{f_s^if_s^j}{1+|\D f_s|^2})\Big(\frac{\D_i\D_j f_s}{\sqrt{1+|\D f_s|^2}} - k_{ij}\Big) = sf_s.
\end{align*}
satisfying $\lim_{x\to \infty} f_s(x) = 0$ at each infinity end. The capillary term $u_s := sf_s$ in regularized equations will play an important role in our analysis. In \cite{SY2} R. Schoen and S.T. Yau proved by maximum principle argument that there are constants $\mu_1 = \max_M |\tr_g k|$ and $\mu_2 = \mu_2(|\mathrm{Ric}|_{C^0(M)}, |k|_{C^1(M)})$ such that in $M$
\begin{align}\label{universal bound of expansion}
    |u_s| = |sf_s|\leq \mu_1, \qquad |\D u_s| = |s\D f_s|\leq \mu_2.
\end{align}
Let $s_j\to 0^+$ be any decreasing sequence such that $f_0 := \lim_{s\to 0^+}f_s$ is a smooth function, and let $\Om_0$, $\Om_+$ and $\Om_-$ be disjoint open sets as stated in Proposition \ref{prop: SY Jang equation}. By Arzela-Ascoli theorem, a subsequence of functions $u_{s_j}$ converges uniformly on $M$ to a Lipschitz function $u\in C^{0,1}(M)$ satisfying 
\begin{align}
    \begin{cases}
    u = 0 \quad \mbox{in $\Om_0$},\\
    u \ge 0 \quad \mbox{in $\Om_+$},\\
    u \le 0 \quad \mbox{in $\Om_-$}.
    \end{cases}
\end{align}
We call $u$ a \emph{capillary blowdown limit} of regularized solutions $f_s$. Let $\Om$ be a connected component of $\Om_+$, which is bounded by Proposition \ref{prop: SY Jang equation}. For simplicity, throughout the present paper we will prove most of the propositions only for connected components of $\Om_+$ and all statements corresponding to $\Om_-$ hold analogously. From now on, we will fix the selection of decreasing sequence $s_j\to 0+$, Lipschitz blow-down limit $u := \lim u_{s_j}$ and connected component $\Om\subset \Om_+$.

Recall that by definition $f_{s_j}\to +\infty$ in $\Om$. In order to study the limit behaviour of $f_{s_j}$ as $j\to \infty$, it is necessary to translate down these regularized solutions in an appropriate manner. It is natural to consider a sequence of reference points $\{x_j\}$ in $\bar{\Om}$ to keep track of the evolution of regularized solutions. For every $j$, we define the translated solution according to the reference point $x_j$ to be
\begin{equation*}
	\tilde{f}_{s_j}^{(x_j)}(\cdot):= f_{s_j}(\cdot)- f_{s_j}(x_j) \quad \mbox{so that} \quad \tilde{f}_{s_j}^{(x_j)}(x_j) = 0.
\end{equation*}
Thus, the regularized equation (\ref{eq: regularized eq}) reads
\begin{align}\label{Eq:TransRegEq}
    \rH(\tilde{f}_{s_j}^{(x_j)}) - \rK(\tilde{f}_{s_j}^{(x_j)}) = s_j f_{s_j}
\end{align}
since the left hand side of regularized equation is invariant under vertical translation. For every sequence $s_j\to 0^+$, the local estimates in Proposition \ref{Prop: LocalEst} and Arzela-Ascoli theorem allow us to find a convergent subsequence of $\gr(\tilde{f}_{s_j}^{(x_j)})$ on the left hand side of (\ref{Eq:TransRegEq}) if we select suitable reference points; the observation (\ref{universal bound of expansion}) and Arzela-Ascoli theorem allow us to find a convergent subsequence of capillary terms (expansion functions) on right hand side of (\ref{Eq:TransRegEq}) in closure of blowup region $\Om$.

The following basic lemma shows that any non-empty subsequential limit must take place in a certain level-set of the capillary blow-down limit $u$.
\begin{lem}\label{Lem:LevelSet}
Suppose the sequence of reference points $\{x_j\}\subset \Om$ converges to $x_0\in \bar{\Om}$. Set the value $\Th: = u(x_0)$. Then
\begin{enumerate}
    \item[(1)] $\Th = \lim u_{s_j}(x_j)$.
    \item[(2)] If $x\in E_\Th^+(u)$, then $\lim \tilde{f}_{s_j}^{(x_j)}(x) = +\infty$; If $x\in E_{\Th}^- (u)$, then $\lim \tilde{f}_{s_j}^{(x_j)}(x) = -\infty$. Therefore, any subsequential limit of graph($\tilde{f}_{s_j}^{(x_j)}$) lies in $E_{\Th}(u)\times \R$ provided it exists. 
\end{enumerate}
\end{lem}

\begin{proof}
(1) It follows immediately from uniform convergence and equicontinuity of $u_{s_j}$ in $\bar{\Om}$.\\
(2) Suppose $x\in E_\Th^+(u)$, then $a := u(x) - \Th>0$. Since $\lim u_{s_j}(x) = u(x)$ and $\lim u_{s_j}(x_j) = \Th$ uniformly, for any sufficiently large $j$
\begin{align*}
    u_{s_j}(x) > u(x) - \frac{a}{4} = \Th + \frac{3a}{4}
\end{align*}
and 
\begin{align*}
    u_{s_j}(x_j) < \Th + \frac{a}{4}.
\end{align*}
If follows that for any sufficiently large $j$
\begin{align*}
    \tilde{f}_{s_j}^{(x_j)}(x) &= \frac{1}{s_j} \big(u_{s_j}(x) - u_{s_j}(x_j)\big)\\
    &> \frac{1}{s_j}[(\Th+\frac{3a}{4}) - (\Th + \frac{a}{4})]\\
    &= \frac{a}{2s_j} \to +\infty.
\end{align*}

If $x\in E_{\Th}^-(u)$, then $\lim_{j\to \infty} \tilde{f}_{s_j}^{(x_j)}(x) =  -\infty$ holds analogously.
\end{proof}

\subsection{The shape of limit of regularized solutions in blowup regions}
In this subsection, we will charaterize the geometry of limits of translated regularized solutions.

In the following theorem we show that any limit graph of caps of $f_{s_j}$ satisfies the \emph{constant expansion equation}, which is an analogue of constant mean curvature equation in spacetime setting. 
\begin{thm}[Shape of cap]\label{thm:interior of level set} Let $\Th :=  \max_{\bar{\Om}} u \geq 0$. There exists a sequence of reference points $\{x_j\}\subset \Om$, a subsequence $\{j'\}\subset \N$, and  a non-empty maximal domain $U\subset u^{-1}(\Theta)\cap \bar{\Om}$ such that $\tilde{f}_{s_{j'}}^{(x_{j'})}$ converges smoothly to a function $\tilde{f}$ in $U$ satisfying the constant expansion equation:
\begin{align}\label{Eq:ConstExp}
    \mathrm{H}(\tilde{f}) - \mathrm{K}(\tilde{f}) = \Theta  \quad \mbox{and $\tilde{f}(x)\to -\infty$ as $U\ni x\to \pa U$}.
\end{align}
Each connected component $\tilde{\Si}$ of $\pa U$ is a closed properly embedded smooth surface in $u^{-1}(\Th) \cap \bar{\Om}$ with constant expansion $\rH_{\tilde{\Si}} - \rK_{\tilde{\Si}}  = \Th$ where $\rH_{\tilde{\Si}}$ is computed with respect to the unit normal of $\tilde{\Si}$ pointing into $U$.
\end{thm}

\begin{rmk}
\begin{enumerate}
	\item[(1)] $U$ is called the {\bf maximal domain} of solution $\tilde{f}$ to constant expansion equation (\ref{Eq:ConstExp}) in the sense that $\tilde{f}$ blows up on approach to $\pa U$ and hence $\tilde{f}$ can not extend to any smooth solution to (\ref{Eq:ConstExp}) defined in a proper superset of $U$.
	\item[(3)] $u$ has constant value $\Th$ in $\bar{U}$.
	\item[(2)] $\Th$ in general could be 0. This means that $u$ is identically 0 in the blowup region $\Om$. We will discuss more properties of $\Om$ in Section 6 when this special case occurs.
\end{enumerate}
\end{rmk}

\begin{proof}
Recall that $\bar{\Om}$ is compact. For every $j\in \N$, pick reference point $x_j\in \bar{\Om}$ such that $f_{s_j}(x_j) = \max_{\bar{\Om}} f_{s_j}$. We can select a convergent subsequence $x_{j'}$ with $x_0\in \bar{\Om}$. Observe that $\tilde{f}_{s_{j'}}^{(x_{j'})}$ is a solution to (\ref{Eq:TransRegEq}) and $(x_{j'}, 0) \in \gr(\tilde{f}_{s_{j'}}^{(x_{j'})})$ converges to $(x_0, 0)$. By the local $C^{3,\al}$-estimate in Proposition \ref{Prop: LocalEst} in a neighborhood of $(x_0, 0)$ and Arzela-Ascoli theorem, we may assume by passing to a further subsequence that $\gr(\tilde{f}_{s_{j'}}^{(x_{j'})})$ converges to a properly embedded submanifold in $C_{loc}^{2,\al}$-sense. Let $\tilde{S}$ denote the connected component of the limit submanifold containing $(x_0, 0)$. Since $\tilde{f}_{s_{j'}}^{(x_{j'})} \leq 0$ in $\bar{\Om}$ for every $j'$, it follows from the Harnack inequality in Proposition \ref{Prop: LocalEst} that the component $\tilde{S}$ is a graph of a $C_{loc}^{2, \al}$ function $\tilde{f}\le 0$ defined in an open neighborhood $U$ of $x_0$ and approaching to $-\infty$ on approach to $\pa U$. Observe that $\lim \tilde{f}_{s_{j'}}^{(x_{j'})}(x) = -\infty$ if $x\in M\bsls \Om_+$, so $U$ is contained in $\Om$ and $x_0$ is away from $\pa\Om$. Combining Lemma \ref{Lem:LevelSet} together with the $C_{loc}^{2, \al}$ convergence of $\tilde{f}_{s_{j'}}^{(x_{j'})}$ and uniform convergence of $u_{j'}$ on two sides of equations (\ref{Eq:TransRegEq}), $U$ is a subset of $u^{-1}(\Th) \cap \Om$ and hence $\tilde{f}$ satisfies equation (\ref{Eq:ConstExp}) in $U$. By standard elliptic theory, $\tilde{f}$ is smooth.

Note that the equation (\ref{Eq:ConstExp}) is invariant under vertical translation. For any $a\in R$, $\tilde{f} + a$ satisfies equation (\ref{Eq:ConstExp}). By local estimates in Proposition \ref{Prop: LocalEst} and Arzela-Ascoli theorem, there is a sequence $a_i \to +\infty$ such that $\gr(\tilde{f} + a_i)$ converge to a three dimensional submanifold in $M\times \R$ in $C^{2,\al}$-sense. Remark that we only need to consider $a\to +\infty$ since $\tilde{f}\le 0$. By the Harnack inequality in Proposition \ref{Prop: LocalEst}, each component of the limit submanifold is a cylinder over a closed surface in $\pa U$, denoted by $\tilde{\Si} \times \R$. Since $\tilde{f} + a_i$ satisfies equation (\ref{Eq:ConstExp}) for all $i$,  $C_{loc}^{2,\al}$-convergence implies that $\tilde{\Si}$ with compatible unit normal satisfies the same constant expansion equation: $\rH_{\tilde{\Si}}- \rK_{\tilde{\Si}} = \Th$. 
\end{proof}

\begin{cor}\label{cor: local extrema}
Let $\Th \ge 0$. Suppose $Z$ is a connected component of $u^{-1}(\Th) \cap \bar{\Om}$ in which $u$ attains local maximum (resp. minimum). Namely, there exists an open neighborhood $O$ of $Z$ such that for all $x\in O\bsls Z$.
\begin{align*}
    u(x)< \Th \quad \mbox{(resp. $u(x) > \Th$).}
\end{align*} 
Then there exists a sequence of reference points $\{x_j\}\subset Z$, a subsequence $\{j'\}\subset\N$ and a non-empty maximal domain $U\subset Z$ such that $\tilde{f}_{s_{j'}}^{(x_{j'})}$ converges smoothly to a function $\tilde{f}$ in $U$ satisfying constant expansion equation:
\begin{align}\label{Eq:ConstExp2}
    \mathrm{H}(\tilde{f}) - \mathrm{K}(\tilde{f}) = \Th \quad \mbox{and $\tilde{f}(x)\to -\infty$ (resp. $+\infty$) as $U\ni x\to \pa U$}.
\end{align}
Each connected connected component $\tilde{\Si}$ of $\pa U$ is a closed properly embedded smooth surface in $Z$ with constant expansion $\rH_{\tilde{\Si}} - \rK_{\tilde{\Si}}  = \Th$ where $\rH_{\tilde{\Si}}$ is computed with respect to the unit normal of $\tilde{\Si}$ pointing inside of $U$ (resp. pointing outside of $U$).
\end{cor} 

\begin{rmk}
In Corollary \ref{cor: local extrema}, the assumption that $u$ attains its strict local maximum $\Th$ in $Z$ is equivalent to that $Z$ is component of $\pa E_\Th^-(u)\bsls \pa E_\Th^+(u)$.
\end{rmk}

\begin{proof}
We only point out the key steps for the case when $u$ attains a strict local maximum in $Z$. Note $Z$ is a closed subset of compact set $\Om$, so $Z$ is compact. We may assume $\bar{O}$ is compact. For every $j\in \N$, pick reference point $x_j\in \bar{O}$ so that $f_{s_j}(x_j) = \max_{\bar{O}} f_{s_j}$. By Lemma \ref{Lem:LevelSet}, $\tilde{f}_{s_j}^{(x_j)}(x)<0$ for all $x\in \pa O$ for all suficiently large $j$. Since $\tilde{f}_{s_j}^{(x_j)}(x_j) = 0$ for all $j$, $\tilde{f}_{s_j}^{(x_j)}$ attains maximum at interior point $x_j\in O$ and $\D\tilde{f}_{s_j}^{(x_j)}(x_j) = 0$ for large $j$. Apply the argument of previous proof, there exists a subsequence $x_{j'}$ converging to $x_0$ and a solution $\tilde{f}$ to (\ref{Eq:ConstExp2}) defined in the maximal domain $U\subset Z$ containing $x_0$. This implies that $x_0$ is an interior point of $Z$ and $x_{j'}\in Z$ for large $j$. Other results follows analogously as previous proof.
\end{proof}

Specifically, we can pick one fixed reference point $x_0\in \bar{\Om}$ and investigate the local limiting behavior of translated regularized solutions to (\ref{eq: regularized eq}) around $x_0$.
\begin{thm}[Local convergence]\label{thm: local convergence}
Let $x_0\in \Om$. Consider the sequence of translated functions $\tilde{f}_{s_j}^{(x_0)}$ satisfying $\tilde{f}_{s_j}^{(x_0)}(x_0) = 0$ for all $j$. There exists a subsequence $\{j'\}\subset \N$ such that one of the following statement is true.
\begin{enumerate}
    \item[(1)] (Graphical convergence) There exists a maximal domain $U_{x_0}\subset u^{-1}(u(x_0)) \cap \bar{\Om}$ containing $x_0$ such that $\tilde{f}_{s_{j'}}^{(x_0)}$ converges smoothly to a function $ \tilde{f}_{0}^{(x_0)}$ satisfying 
    \begin{align*}
        \rH(\tilde{f}_{0}^{(x_0)})-\rK(\tilde{f}_{0}^{(x_0)}) = u(\bar{U}_{x_0})\quad \mbox{and $|\tilde{f}_0^{(x_0)}|\to \infty$ on approach to $\pa U_{x_0}$}.
    \end{align*} 
    In particular, 
    \begin{align*}
        \lim_{j'\to \infty} |\D f_{s_{j'}}(x_0)| = \lim_{j'\to \infty} |\D\tilde{f}_0^{(x_0)}(x_0)| < +\infty,
    \end{align*}
    and
    \begin{align*}
        \lim_{j'\to \infty} \frac{\D f_{s_{j'}}(x_0)}{\sqrt{1 + |f_{s_{j'}}(x_0)|^2}} \quad \mbox{exists and has length $< 1$.}
    \end{align*}
    Each component $\Si$ of $\pa U_{x_0}$ is a closed smooth surface satisfying 
    \begin{align*}
        \rH_{\Si} - \rK_{\Si} = u(\bar{U}_{x_0}).
    \end{align*}
    Here, $\rH_\Si$ is computed with respect to the unit normal vector field $\nu$ which coincides with 
    \begin{align*}
        \nu(y) = \lim_{j'\to \infty} \frac{\D f_{s_{j'}}(y)}{\sqrt{1 + |\D f_{s_{j'}}(y)|^2}} \quad \mbox{for all $y\in \Si$.}
    \end{align*}

    \item[(2)] (Cylindrical convergence) There exists a closed smooth surface $\Si_{x_0} \subset u^{-1}(u(x_0)) \cap \bar{\Om}$ passing through $x_0$ such that $\gr(\tilde{f}_{s_{j'}}^{(x_0)})$ converges to $\Si_{x_0}\times \R$ smoothly. In particular, 
    \begin{align*}
        \lim_{j'\to \infty} |\D f_{s_{j'}}(x_0)| = +\infty,
    \end{align*}
    and 
    \begin{align*}
        \lim_{j'\to \infty} \frac{\D f_{s_{j'}}(x_0)}{\sqrt{1 + |\D f_{s_{j'}}(x_0)|^2}} \quad \mbox{exists and has length $= 1$.}
    \end{align*}
    The surface $\Si_{x_0}$ satisfies 
    \begin{align*}
        \rH_{\Si_{x_0}} - \rK_{\Si_{x_0}} = u(\Si_{x_0}).
    \end{align*}
    Here, $\rH_{\Si_{x_0}}$ is computed with respect to the unit normal vector field $\nu$ which coincides with
    \begin{align*}
        \nu(y) = \lim_{j'\to \infty} \frac{\D f_{s_{j'}}(y)}{\sqrt{1 + |\D f_{s_{j'}}(y)|^2}} \quad \mbox{for all $y\in \Si_{x_0}$.}
    \end{align*}
    As a consequence of the convergence, there exists $\de >0$ such that
    \begin{enumerate}
        \item[(a)]  $\lim_{j'\to \infty} \tilde{f}_{s_{j'}}^{(x_0)}(x) = +\infty$ for $x\in \cN_{\de}^+(\Si_{x_0}, \nu)$,
        \item[(b)] $\lim_{j'\to \infty} \tilde{f}_{s_{j'}}^{(x_0)}(x) = -\infty$ for $x\in \cN_{\de}^-(\Si_{x_0}, \nu)$.
    \end{enumerate}
\end{enumerate}
\end{thm}

\begin{proof}
Since $\tilde{f}_{s_j}^{(x_0)}$ satisfies (\ref{Eq:TransRegEq}) and $(x_0, 0)\in \gr(\tilde{f}_{s_j}^{(x_0)})$ for all $j$, by the local estimate in Proposition \ref{Prop: LocalEst} we conclude that there exists a subsequence $\{j'\}$ such that graph$(\tilde{f}_{s_{j'}}^{(x_0)})$ converges to a properly embedded submanifold in $M\times \R$ in $C_{loc}^{2,\al}$-sense. Denote the component of the limit submanifold containing $(x_0, 0)$ by $\tilde{S}$. By the Harnack-type inequality in Proposition \ref{Prop: LocalEst}, $\tilde{S}$ is either graphical or cylindrical. 

Notice that $\D\tilde{f}_{s_{j}}^{(x_0)}(x) = \D f_{s_{j}}(x)$ for all $j\in \N , x\in M$, and the vector 
\begin{equation*}
    \frac{\D f_{s_j}}{\sqrt{1+ |\D f_{s_j}|^2}}
\end{equation*}
is the horizontal component of the downward unit normal vector field on $\gr(\tilde{f}_{s_j}^{(x_0)})$. For the graphical case, the results follow analogously as the proof of Theorem \ref{thm:interior of level set}. For the cylindrical case, by Lemma \ref{Lem:LevelSet} and $C_{loc}^{2, \al}$ convergence we have $\rH_{\Si_{x_0}} - \rK_{\Si_{x_0}} = u(x_0)$. Lastly, to check the compatibility of the unit normal of $\pa U_{x_0}$ or respectively $\Si_{x_0}$ at $y$, we may just pick $y$ as new reference point for the subsequence $f_{s_{j'}}$, then the limiting behavior of $f_{s_{j'}}^{(x_0)}$ near $y$, local estimate and Harnack inequality imply that any subsequence of $\gr(\tilde{f}_{s_{j'}}^{(y)})$ converges cylindrically to the component of $\pa U_{x_0} \times \R$ containing $(y, 0)$ or respectively $\Si_{x_0} \times \R$ and the original sequence converges in the same way.
\end{proof}

As an immediate application of the local convergence, we can show the existence of smooth closed constant expansion surface in any level set of $u$ containing a regular point.
\begin{cor}\label{cor:CES in level set} Suppose $x_0$ is a regular point of $u$ at value $\Th$. Then there exists a closed smooth embedded surface $\Si_{x_0}$ in $u^{-1}(\Th) \cap \bar{\Om}$ containing $x_0$ with constant expansion $\rH_{\Si_{x_0}} - \rK_{\Si_{x_0}} = \Th$. The unit normal vector field $\nu$ of $\Si_{x_0}$ chosen as in Theorem \ref{thm: local convergence} coincides with $\D u(x_0)/|\D u(x_0)|$ at $x_0$.
\end{cor}

\begin{proof}
From assumption, $\D u(x_0)$ exists and $\D u(x_0)\neq 0$. By local linear approximation of $u$ around $x_0$, we can conclude that $x_0\in \pa E_\Th^-(u) \cap \pa E_\Th^+(u)$ and the tangent space $T_{x_0} E_\Th (u) = \{\D u(x_0)\}^{\perp}$. Since $x_0$ is not an interior point of $u^{-1}(\Th)$, Corollary \ref{thm: local convergence} implies that graphical convergence is impossible and there exists a closed smooth CES $\Si_{x_0}$ containing $x_0$ with expansion $\Th$. The direction of unit normal to $\Si_{x_0}$ at $x_0$ is determined by local linear approximation of $u$, Lemma \ref{Lem:LevelSet} and local limit behavior of $f_{s_j}^{(x_0)}$ in Theorem \ref{thm: local convergence} (2).
\end{proof}

\subsection{Stability of CES}
We will end this section by showing the stability of all closed smooth embedded CES in $M$ as boundary components of maximal domains or base sections of cylinders in subsection 3.2. The stability result for MOTS Proposition \ref{prop: stability for MOTS} was first proved by Andersson and Metzger in \cite{AM}. In the excellent survey paper \cite{AEM}, a simplified geometric argument was provided, but the constructed barrier functions did not work well. In the communication with Michael Eichmair, one of the authors of \cite{AEM}, he suggested a different model function to fix the glitch. The proof here essentially follows the idea for MOTS in \cite{AEM} with Eichmair's modification. 
\begin{prop}[Stability of CES]\label{prop: stability of CES}
The closed smooth CES which arise in Theorem \ref{thm:interior of level set}, Corollary \ref{cor: local extrema}, and Theorem \ref{thm: local convergence} (1) as boundaries of maximal domains and in Theorem \ref{thm: local convergence} (2) as bases of cylinders are stable.
\end{prop} 

\begin{proof}
Suppose $(\Si, \nu)$ is a \emph{unstable} closed smooth surface in $(M, g, k)$ with constant expansion $\Th$ and $\la_1(\cL_\Si) =  -\alpha^2 <0$ for some $\al>0$. We will construct barrier functions in an open neighborhood of $\Si$. By Krein-Rutman theorem, there exists a strict positive function $\phi\in C^\infty(\Si)$ such that $\cL_\Si \phi = -\al^2 \phi$. If $\nu$ is extended by parallel transportation and $k$ is extended trivially in vertical direction, then the stability operator of $(\Si\times \R, \nu)$ with respect to $(M,g,k)$ is $\cL_{\Si\times \R} = -\pa_t^2 + \cL_\Si$. If we feed the stability operator with a test function of the form $T(t)\phi(x)$ where $T\in C^2(\R)$, then 
\begin{equation*}
	\cL_{\Si\times \R} \big(T(t)\phi(x)\big) =  -(T'' + \alpha^2 T)\phi(x).
\end{equation*}
In the model case $\al = 1$, consider the smooth function $\eta(t) =  \big(\arctan(t+1) - \arctan(1)\big)$. By numerical analysis, $\eta$ has the following properties: (1) $\mathrm{Range}(\eta) = (-\frac{3\pi}{4}, \frac{\pi}{4})$ and $\eta$ is strictly increasing with $\eta(0) = 0$. (2) $\eta'' + \eta$ has a unique real root $t_r \approx 0.6456$. In particular, for $t \in (-\infty, 1/2]$, $\eta''(t) + \eta(t) <0$ and the maximum is $\eta''(1/2)+\eta(1/2)\approx -0.0866$. For general $\al>0$, we may consider $T(t) = \eta(\al t)$. Then \begin{equation}
	\cL_{\Si\times \R} \big(T(t)\phi(x)\big)\geq -\big(\eta''(1/2)+\eta(1/2)\big)\al^2\min_\Si\phi >0 \quad \mbox{for}\quad t\in (-\infty, 1/(2\al)].
\end{equation}
For any sufficiently small $\vare>0$, the hypersurface 
\begin{equation*}
	\big\{\exp_{(x,t)}\big(\vare T(t)\phi(x)\nu(x)\big)\in M\times \R: (x,t)\in  \Si\times(-\infty, 1/(2\al)]\big\}
\end{equation*}
is a smooth hypersurface with boundary in $M\times \R$ whose expansion is \emph{strictly} greater than $\Th$ everywhere. Notice that $T$ is monotone, so the hypersurface can be express as the graph of a function $f_*: V\rightarrow (-\infty, 1/(2\al))$
where $V = \big\{\exp_x(s\phi(x)\nu(x))\in M: s\in (-3\pi\vare/4, \vare\eta(1/2)), x\in \Si\big\}$ is a open neighborhood of $\Si$ such that $0<f_*<1/(2\al)$ in the part of $V$ with $0<s<\vare\eta(1/2)$ and $f_*\to -\infty$ as $s\to -3\pi\vare/4^+$. Moreover, $f_*$ satisfying $\rH(f_*) -  \rK(f_*) > \Th$ is a sub-solution to equation $\rH-\rK=\Th$. Analogously, we may construct a super-solution $f^*$ satisfying $\rH(f^*)-\rK(f^*) < \Th$ associated to the hypersurface
\begin{equation*}
	\big\{\exp_{(t,x)}\big(-\vare T(-t)\phi(x)\nu(x)\big): (x,t)\in \Si\times [-1/(2\al),+\infty)\big\}.
\end{equation*}
defined in an open neighborhood of $\Si$.

If the CES $(\Si, \nu)$ which arises in the regularization limit as in the assumption is unstable, then the barrier functions constructed in the first paragraph prevent the translated regularized solutions from blowing up exactly at $\Si$. This contradicts to the formation of such CES.
\end{proof}

\section{Characterization of capillary blowdown limit}
\subsection{Capillary blowdown limit as viscosity solution}
We begin by replacing $f_s$ by $u_s/s$ in regularized equations (\ref{eq: regularized eq}). Then $u_s$ satisfies
\begin{align}\label{regeq2}
    (g^{ij}- \frac{u_s^iu_s^j}{s^2+|\D u_s|^2})\Big(\frac{\D_i\D_j u_s}{\sqrt{s^2+|\D u_s|^2}} - k_{ij}\Big) = u_s.
\end{align}
Let $u$ be a blowdown limit of regularized solutions to Jang equation. Now we make some heuristic assumptions that $u$ is $C^2$ and $\D u_{s_j}\to \D u$ for one sequence $s_j\to 0^+$. In the region $\{x: \D u(x)\neq 0\}$, the sequence of regularized equations (\ref{regeq2}) converges to the geometric equation
\begin{align}\label{Eq:ConstExpLevel}
    \Div_M \Big(\frac{\D u}{|\D u|}\Big) - \tr_g (k) + k\Big(\frac{\D u}{|\D u|}, \frac{\D u}{|\D u|}\Big) = u.
\end{align}
In addition, Corollary \ref{cor:CES in level set} is another clue that $u$ satisfies (\ref{Eq:ConstExpLevel}) in $\{x: \D u(x)\neq 0\}$. This geometric equation can be interpreted as follows: any regular level set of classical solution $u$ has constant expansion equal to the evaluation of $u$. By simple calculations, (\ref{Eq:ConstExpLevel}) is equivalent to 
\begin{align}\label{Eq:ConstExpLevel2}
    -\Div_M (\D u) + \D^2u\Big(\frac{\D u}{|\D u|},\frac{\D u}{|\D u|}\Big) + |\D u|\Big\{ u + \tr_g (k) - k\Big(\frac{\D u}{|\D u|}, \frac{\D u}{|\D u|}\Big)\Big\} = 0.
\end{align}
It is obvious that the equation is singular in the set $\{\D u = 0\}$, which is inevitable according to the existence of interior of level set at extremal value by Theorem \ref{thm:interior of level set}. It is necessary to find a weaker notion of solution. Since $u$ has already been a Lipschitz continuous function by construction, inspired by the work on level-set formulation of mean curvature flow done by L.C. Evans and J. Spruck \cite{ES}, we may expect viscosity solution is suitable notion of weak solution. Before we define viscosity solutions to (\ref{Eq:ConstExpLevel2}) on manifolds, we recall several terminologies introduced in \cite{AFS}.
\begin{defn}
\begin{enumerate}
    \item[(1)] Let $f:M \rightarrow [-\infty, \infty) $ a lower semi-continuous function. Define the \emph{second order superjet} of $f$ at $x$ by
\begin{align*}
    J^{2,+} f(x) = \{(d\varphi(x), d^2\varphi(x)):\mbox{ $\varphi \in C^2(M;\R)$, $f-\varphi$ attains a local maximum at $x$}\}
\end{align*}

    \item[(2)] Let $f:M \rightarrow (-\infty, \infty] $ a upper semi-continuous function. Define the \emph{second order subjet} of $f$ at $x$ by
\begin{align*}
    J^{2,-} f(x) = \{(d\varphi(x), d^2\varphi(x)):\mbox{ $\varphi \in C^2(M;\R)$, $f-\varphi$ attains a local minimum at $x$}\}
\end{align*}
\end{enumerate}

\end{defn}

\begin{rmk}
Let $x\in M$, $\zeta \in T_x^*M$, $A\in \mathcal{L}_{sym}^2(T_xM)$. Then the followings are equivalent:
\begin{enumerate}
    \item[(1)] $(\zeta, A)\in J^{2,+} f(x)$
    \item[(2)] $f(\exp_x(\eta)) \leq f(x) + \lan \zeta, \eta \ran_x + \frac{1}{2} \lan A \eta, \eta\ran_x + o(|\eta|_x^2)$
    \item[(3)] $(\zeta, A) \in J^{2,+} (f\circ \exp_x) (0_x)$ where $0_x$ is the origin in $T_x M$
    \item[(4)] $(\zeta, A) \in -J^{2,-} (-f)(x)$
\end{enumerate}
\end{rmk}

Let $x_n\to x$, $\zeta_n\in T_{x_n}^*M$ and $A_n\in \mathcal{L}_{sym}^2(T_{x_n}M)$. We denote by $\zeta_n \to \zeta\in T_x^*M$ if $\lan \zeta_n, V\ran_{x_n}\to \lan \zeta, V\ran_x$ for all smooth vector field $V$ near $x$ and we denote by $A_n\to A\in \mathcal{L}_{sym}^2(T_xM)$ if $\lan A V, V\ran_{x_n} \to \lan A V, V\ran_x$ for all smooth vector field $V$ near $x$.

\begin{defn}
\begin{enumerate}
    \item[(1)] Let $f:M \rightarrow [-\infty, \infty) $ a lower semi-continuous function. Define
\begin{align*}
    \bar{J^{2,+}} f(x) = \{(\zeta, A)\in T_x^*M \times \mathcal{L}_{sym}^2(T_x M): \exists x_n\to x, \exists (x_n, A_n)\in J^{2,+} f(x_n) \\\mbox{ such that } (x_n, f(x_n), \zeta_n, A_n) \to (x, f(x), \zeta, A)\}
\end{align*}

    \item[(2)] Let $f:M \rightarrow (-\infty, \infty] $ a upper semi-continuous function. Define 
\begin{align*}
    \bar{J^{2,-}} f(x) = \{(\zeta, A)\in T_x^*M \times \mathcal{L}_{sym}^2(T_x M): \exists x_n\to x, \exists (x_n, A_n)\in J^{2,-} f(x_n) \\\mbox{ such that } (x_n, f(x_n), \zeta_n, A_n) \to (x, f(x), \zeta, A)\}
\end{align*}
\end{enumerate}
\end{defn}

Now we are ready to define viscosity solutions to (\ref{Eq:ConstExpLevel2}). Let $x\in M$, $r\in \R$, $\zeta\in T_xM$, $A\in \mathcal{L}_{sym}^2(T_xM)$. Define
\begin{align*}
    \cF(x, r, \zeta, A) := -\tr_g A(x) + \lan A \frac{\zeta}{|\zeta|}, \frac{\zeta}{|\zeta|} \ran_x + |\zeta|_x \big\{ r + \tr_g k(x) - k(\frac{\zeta}{|\zeta|},\frac{\zeta}{|\zeta|})(x)\big\}
\end{align*}
and its degenerate form
\begin{align*}
    \cG(x, \zeta, A) :=  -\tr_g A(x) + \lan A\zeta, \zeta \ran_x.
\end{align*}

\begin{defn}
$u\in C^0(M)\cap L^\infty(M)$ is a \textbf{\emph{viscosity subsolution}} of equation (\ref{Eq:ConstExpLevel2}) if for all $x\in M$ either for all $(\zeta\neq 0, A)\in \bar{J^{2,+}} u(x)$
\begin{align*}
    \cF(x, u(x), \zeta, A) \leq 0,
\end{align*}
or for all $(0, A)\in \bar{J^{2,+}} u(x)$ there exists $\xi\in T_xM$ with $|\xi|_x \leq 1$
\begin{align*}
    \cG(x, \xi, A) \leq 0.
\end{align*}
Similarly, $u\in C^0(M)\cap L^\infty(M)$ is a \textbf{\emph{viscosity supersolution}} of equation (\ref{Eq:ConstExpLevel2}) if for all $x\in M$ either for all $(\zeta\neq 0, A)\in \bar{J^{2,-}} u(x)$ 
\begin{align*}
    \cF(x, u(x), \zeta, A) \geq 0,
\end{align*}
or for all $(0, A)\in \bar{J^{2,-}} u(x)$ there exists $\xi\in T_xM$ with $|\xi|_x \leq 1$
\begin{align*}
    \cG(x, \xi, A) \geq 0.
\end{align*}
$u\in C^0(M)\cap L^\infty(M)$ is a \textbf{\emph{viscosity solution}} of equation (\ref{Eq:ConstExpLevel2}) if $u$ is both a viscosity subsolution and supersolution.
\end{defn}

In the following theorem, we apply the argument in the proof of existence of weak mean curvature flow in viscosity sense using elliptic regularization by L.C. Evans and J. Spruck \cite{ES} to show that any blowdown limit of regularized solutions is a viscosity solution.
\begin{thm}\label{thm:viscosity solution}
Let $u$ be a capillary blowdown limit of $f_s$. Then $u$ is a viscosity solution to the geometric equation (\ref{Eq:ConstExpLevel2}).
\end{thm}
\begin{proof}
Let $\varphi\in C^2(M)$ and suppose $u-\varphi$ has a \emph{strict} local maximum at a point $x_0\in M$. Choose $u_{s_j} \to u$ uniformly near $x_0$, then $u_{s_j}-\varphi$ has a local maximum at a point $x_j$ with $x_j\to x_0$ as $j\to \infty$. Since $u_{s_j}$ and $\varphi$ are twice differentiable, we have 
\begin{align*}
    \D u_{s_j}(x_j) &= \D\varphi(x_j),\\
    \D^2 u_{s_j}(x_j) &\leq \D^2\varphi(x_j).
\end{align*}
Thus, equation (\ref{regeq2}) implies for all $j$ at $x_j$
\begin{align}
    &-\tr_g \D^2 \varphi + \D^2 \varphi \big(\frac{\D\varphi}{\sqrt{s_j^2+ |\D\varphi|^2}}, \frac{\D\varphi}{\sqrt{s_j^2+ |\D\varphi|^2}}\big) \notag\\
    &+ \sqrt{s_j^2+ |\D\varphi|^2}\Big\{u_{s_j} + \tr_g k - k\big(\frac{\D\varphi}{\sqrt{s_j^2+ |\D\varphi|^2}},\frac{\D\varphi}{\sqrt{s_j^2+ |\D\varphi|^2}}\big)\Big\} \leq 0 \label{Eq2}
\end{align}

Suppose $\D\varphi(x_0) \neq 0$. Then $\D\varphi(x_j) \neq 0$ for all sufficiently large $j$. Passing to limit, we get
\begin{align*}
    \cF\big(x_0, u(x_0), \D\varphi(x_0), \D^2 \varphi(x_0)\big)\leq 0.
\end{align*}
Suppose $\D\varphi(x_0) = 0$. Set $\eta_j := \frac{\D\varphi(x_j)}{\sqrt{s_j^2 + |\D\varphi(x_j)|^2}}\in T_{x_j}M$ such that (\ref{Eq2}) becomes
\begin{align*}
    &-\tr_g \D^2 \varphi(x_j) + \D^2 \varphi \big(\eta_j, \eta_j\big)(x_j) \\
    &+ \sqrt{s_j^2+ |\D\varphi(x_j)|^2}\big\{u_{s_j}(x_j) + \tr_g k(x_j) - k\big(\eta_j,\eta_j\big)(x_j)\big\} \leq 0
\end{align*}
Since $|\eta|_{x_j} \leq 1$, we may assume up to subsequence $\eta_j \to \eta\in T_{x_0}M$ with $|\eta|_{x_0} \leq 1$. Letting $j\to \infty$, since $u$ and $k$ are bounded we obtain
\begin{align*}
    \cG(x_0, \eta, \D^2 \varphi(x_0)) \leq 0.
\end{align*}

If $u-\varphi$ has a local maximum which may not be strict, we repeat the argument above with  
\begin{align*}
    \Tilde{\varphi}(x) = \varphi + d(x,x_0)^4
\end{align*}
satisfying $\D\tilde{\varphi}(x_0) = \D\varphi(x_0)$ and $\D^2\tilde{\varphi}(x_0) = \D^2 \varphi(x_0)$ in place of $\varphi$. Here, $d$ is the distance function defined on $(M, g)$. Therefore, $u$ is a viscosity subsolution.

It follows analogously that $u$ is a viscosity supersolution.
\end{proof}

\subsection{A priori estimates of foliation of stable constant expansion surfaces}
In this subsection, we will prove the a priori estimate of foliation of stable constant expansion surfaces. The proof will follow the stability argument leading to the a priori estimates of regularized Jang equation in \cite{SY2} and the one of stable minimal hypersurfaces in \cite{SSY}. Here, we only comment on the key ingredients adapted to the assumptions that we concern. 

Recall that $(M,g,k)$ is an asymptotically flat initial data set satisfying dominant energy condition. Given positive constants $\cT$ and $B$, suppose $\Si$ assigned with unit normal $\nu$ is a closed smooth stable CES with constant expansion $\th_0\in [-\cT, \cT]$ having the second fundamental forms $\|h_{\Si}\|^2\leq B$. Suppose $\Psi: (a, b)\times \Si \rightarrow M$ is a smooth foliation of closed stable CES initiated from $\Si$ with expansion in the range $[-\cT, \cT]$. Let $\Si_\tau$ denote $\Psi(\tau, \Si)$ and let $\nu_\tau$ denote $\Psi_*(\pa_\tau)/|\Psi_*(\pa_\tau)|$ where $\Psi_*$ is the pushforward of $\Psi$, then $\Psi$ satisfies the following properties:
\begin{enumerate}
	\item[(1)] $\Psi(\tau_0, \cdot) = \mathrm{Id}_\Si(\cdot)$ on $\Si$ for some $\tau_0\in (a,b)$ and $\nu_{\tau_0} = \nu$.
    \item[(2)] The expansion $\th_{\Si_\tau}$ of $\Si_\tau$ with respect to unit normal $\nu_\tau$ is a constant in $[-\cT, \cT]$ for any $\tau\in (a,b)$.
    \item[(3)] $\la_1(\cL_{\Si_\tau}) \geq 0$ for any $\tau\in (a , b)$. 
\end{enumerate}
Now fix arbitrary $\tau \in (a, b)$. Let $e_1, e_2, e_3$ be a local orthonormal frame for $\Si_\tau$ with $e_1,e_2$ tangent to $\Si_\tau$ and $e_3$ normal to $\Si_\tau$. Let $\om_1, \om_2, \om_3$ be the corresponding dual orthonormal coframe of one-form. The structure equations of $M$ is given by 
\begin{align*}
    d\om_i = -\sum_{j=1}^3 \om_{ij}\wedge \om_j ,\qquad \om_{ij}+\om_{ji} = 0\\
    d\om_{ij} = -\sum_{k = 1}^{3} \om_{ik}\wedge \om_{kj} + \frac{1}{2}\sum_{k,l = 1}^3 R_{ijkl}\,\om_k\wedge \om_l.
\end{align*}
Let $\D$ and $\bar{\D}$ denote the Levi-Civita connections on $M$ and $\Si_\tau$ respectively. In this subsection, the indices range 1,2 and constant $C$ may change from time to time but depend only on the initial data set $(M,g,k)$, given constants $\cT$ and $B$. The first important ingredient of a priori estimate is Simon's inequality. By virtue of asymptotically-flatness of $(M,g,k)$, the background Riemannian curvature tensor and its covariant derivatives are bounded. This is a key assumption to derive the lower bound of Laplacian of second fundamental form $h_{ij}$ of $\Si_\tau$ as in \cite{SY2} on page 236
\begin{align*}
    \De h_{ij} \geq \bar{\D}_i\bar{\D}_j \rH - (\sum_{m,k} h_{mk}^2) + \rH \sum_{m} h_{im}h_{mj} - C (|h| + 1)\de_{ij}.
\end{align*}
Follwing the same computation in \cite{SY2} page 236-237, one can obtain the Simon's inequality (cf. \cite{SY2} (2.16))
\begin{align}\label{simon inequality}
    |h| \De |h| &\geq  c(2) \sum_{i,j,k} (\bar{\D}_k h_{ij})^2 - |h|^4 - |\rH||h|^3\notag\\
    & +\sum_{i,j} h_{ij}\bar{\D}_i\bar{\D}_j \rH - C|\bar{\D}\rH|^2 - C(|h|^2 + 1)
\end{align}
where $c(2)$ is a constant depends only on $\mathrm{dim}(\Si) = 2$.

The second important ingredient is the stability inequality. In \cite{SY2}, they derive the stability inequality by observing that vertical translations generate a Jacobi field. Now in our setting we assume the stability directly. Let $\beta > 0$ be a smooth eigenfunction of $\cL_{\Si_\tau}$ corresponding to non-negative principle eigenvalue $\la_1$ and let $\th_\tau$ be the constant expansion of $\Si_\tau$. Using (\ref{op2}), we have
\begin{align}\label{stability inequality 1}
\begin{split}
    0\leq \la_1 = \frac{\cL_{\Si_\tau} \beta}{\beta} &= -\Div_{\Si_\tau}(\xi + \bar{\D} \log \beta) - |\xi + \bar{\D}\log \beta|_{\Si_\tau}^2\\
    &+ \frac{1}{2}\rR_\Si - \frac{1}{2}|h - k|_{\Si_\tau}^2 - \mu + J(\nu) - \frac{1}{2}\th_\tau(\th_\tau + 2\tr_g k).
\end{split}
\end{align}
Note that dominant energy condition implies $-\mu + J(\nu)\leq 0$. Let $\varphi\in C^\infty(\Si)$. Multiplying (\ref{stability inequality 1}) by $\varphi^2$, integrating by part and applying Young's inequality to the first term, we find
\begin{align}\label{stability inequality 2}
    0 &\leq  \int_{\Si_\tau} |\bar{\D}\varphi|^2 + \frac{1}{2}\int_{\Si_\tau} \Big\{\rR_{\Si_\tau} - |h-k|_{\Si_\tau}^2 - \th_\tau(\th_\tau + 2\tr_g k)\Big\}\varphi^2.
\end{align}
Using Guass equation and cancelling out $\rH^2$ terms in $\rR_\Si$ and $\th_\tau^2$, we get
\begin{align}\label{h^2 bound}
    \int_{\Si_\tau} |h|^2\varphi^2 \leq \int_{\Si_\tau} |\bar{\D}\varphi|^2 + C \int_{\Si_\tau} (|h| + 1)\varphi^2.
\end{align}
Combining (\ref{simon inequality}), (\ref{h^2 bound}) together with the control $|\bar{\D}\rH_\Si|^2 = |\bar{\D} \rK_\Si|^2\leq C(|h|^2 +1)$ on constant expansion surface, following the argument in \cite{SY1} replacing $\varphi$ by $|h|\varphi^2$ and then absorbing $|h|^3\varphi^4$ by $|h|^4\varphi^4$ and $\varphi^2$, we may derive 
\begin{align}\label{h^4 bound}
    \int_{\Si_\tau} |h|^4\varphi^4 \leq \int_{\Si_\tau} |\bar{\D}\varphi|^4 + C \int_{\Si_\tau} \varphi^4.
\end{align}

The third ingredient is the local area bound for $\Si_\tau$. We will follow the calibration argument in \cite{SY1} on page 243 with minor modification. Observe that in the region sweep by the foliation $\Psi$ we have
\begin{equation}\label{calibration}
	\Div_M(\nu_\tau) = \th_{\Si_\tau} + \tr_g(k) - k(\nu_\tau, \nu_\tau)
\end{equation}
where $|\th_{\Si_\tau}|\leq \cT$. Let $x_0\in \Si_\tau$, $B_\si(x_0)$ be the geodesic ball in $(M,g)$ centered at $x_0$ and let $W$ be the region enclosed by $\Si$ and $\Si_\tau$. Let $0<\rho_0\leq 1$ such that $\rho_0\leq \mathrm{inj}(M, g)$. Integrating identity (\ref{calibration}) over the region $W\cap B_\si(x_0)$ for $0<\si\leq \rho_0$ and applying divergence theorem, we obtain
\begin{equation*}
	\area(\Si_\tau\cap B_\si(x_0)) \leq \area(\Si \cap B_\si(x_0)) + \area(\pa B_\si(x_0)\cap W) + C\cT \si^3.
\end{equation*}
Since we have $\|h_\Si\|^2\leq B$ for the initial sheet, there exists a constant $\rho_1$ depending on $M,\Si, g,k,B,\cT$ such that for $0<\si\leq \rho_1$
\begin{equation}\label{area bound}
	\area(\Si_\tau\cap B_\si(x_0)) \leq C \si^2.
\end{equation}
With the area bound (\ref{area bound}) the results of Hoffman and Spruck \cite{HS} imply that there is a number $\rho_2\leq \rho_1$ such that the Michael-Simon type Sobolev inequality holds: 
\begin{equation}
	\Big(\int_{\Si_\tau}\varphi^2\Big)^{1/2} \leq C \int_{\Si_\tau} |\bar{\D}\varphi| + |\varphi||\rH|.
\end{equation}
for any Lipschitz $\varphi$ vanishing outside of $\Si_\tau \cap B_{\rho_2}(x_0)$. Using the bounds for expansion, $k$ and area (\ref{area bound}) together with H\"{o}lder inequality, we obtain
\begin{equation*}
	\Big(\int_{\Si_\tau}\varphi^2\Big)^{1/2} \leq C \int_{\Si_\tau} |\bar{\D}\varphi|
\end{equation*}
and hence for arbitrary $p>2$
\begin{equation}\label{Sobolev inequality}
	\Big(\int_{\Si_\tau} |\varphi|^p\Big)^{1/p} \leq C \int_{\Si_\tau} |\bar{\D}\varphi|^2.
\end{equation}
Fixing the geodesic distance cutoff function to $x_0$ depending on $\rho_2$, (\ref{h^4 bound}) and (\ref{area bound}) imply
\begin{equation}\label{L2 bound for h}
	\|h_{\Si_\tau}\|^2\in L^2\big(B_{\rho_2/2}(x_0)\big).
\end{equation}
Let $u = \|h_{\Si_\tau}\|^2 + 1$. Following the argument in \cite{SY1} $u$ is a positive weak subsolution to some elliptic equation. De Georgi-Nash-Moser iteration technique together with the $L^2$-bound for $u$ (\ref{area bound}) and (\ref{L2 bound for h} now gives pointwise curvature bound for extrinsic curvature
\begin{equation}\label{pointwise bound for h}
	\sup_{\Si_\tau} |h_{\Si_\tau}|^2 \leq C.
\end{equation}
Note that the Sobolev inequality (\ref{Sobolev inequality}) for large $p>2$ is sufficient for iteration technique for dimension 2. Also, (\ref{L2 bound for h}) where $2> \frac12 \dim(\Si_\tau) = 1$ guarantees the structural conditions are satisfied. 

Lastly, following the argument in \cite{SY1}, (\ref{pointwise bound for h}) implies the uniform local $C^{3,\al}$ estimate. We conclude the results of this subsection in the following proposition.
\begin{prop}\label{prop: Local est for foliation of CES} Let $(M,g,k)$ be an asymptotically flat initial data set satisfying dominant energy condition. Given positive constants $\cT$ and $B$, suppose $\Si$ assigned with unit normal $\nu$ is a closed smooth stable CES with constant expansion $\th_0\in [-\cT, \cT]$ having the second fundamental forms $\|h_{\Si}\|^2\leq B$. Suppose $\Psi: (a, b)\times \Si \rightarrow M$ is a smooth foliation of closed stable CES initiated from $\Si$ with expansion in the range $[-\cT, \cT]$. Given $\al\in (0,1)$, then there exist constants $\rho$ and $C_\al$ depending on $M,\Si, g, k, \cT, B$ such that for any $\tau \in (a,b)$, for every $x_0\in \Si_\tau$ if $(x^1,x^2,x^3)$ normal coordinates in $M$ on which $T_{x_0} \Si_\tau$ is the $x^1x^2$-space, then the local defining function $w(x)$ for $\Si_\tau$ is defined on $\{x=(x^1,x^2): |x|\leq \rho\}$ with 
\begin{align*}
    \Si_\tau\cap B^3 (x_0;\frac{\rho}{2}) \subseteq \gr(w)
\end{align*} and satisfies 
\begin{align*}
    \norm{w}_{3,\al,\{x:|x|\leq \rho\}} \leq C_\al.
\end{align*}
\end{prop}

\subsection{Existence of smooth solutions}
\begin{prop}\label{prop: local stable foliation}
Suppose $(\Si, \nu)$ is a closed smooth embedded strictly stable CES in $(M,g,k)$ with $\th_\Si \equiv  \tau_0$ in $(M,g,k)$. Then there exists a constant $\vare>0$ and a smooth CES foliation $\Psi: (\tau_0 - \vare, \tau_0 + \vare)\times \Si \rightarrow M$ satisfying the following properties. Let $\Si_\tau$ denote the sheet $\Psi(\tau, \Si)$. We have
\begin{enumerate}
    \item[(1)] $\Psi(\tau_0, \cdot) = \mathrm{Id}_\Si(\cdot)$ on $\Si$.
    \item[(2)] $\th_{\Si_\tau} \equiv \tau$ for all $\tau\in (\tau_0 - \vare, \tau_0 + \vare)$.
    \item[(3)] (Local uniqueness) If $\tilde{\Si}$ is a closed smooth CES in $\Psi\big((\tau_0 - \vare, \tau_0 + \vare)\times \Si\big)$ and can be expressed as a graph of $w\in C^\infty(\Si)$ in Fermi coordinates around $\Si$, then $\tilde{\Si} = \Si_{\tilde{\tau}}$ for some $\tilde{\tau}\in (\tau_0 - \vare, \tau_0 + \vare)$.
\end{enumerate}
\end{prop}

\begin{proof}
We begin with proving local existence of smooth foliation by using implicit function theorem. Let $\Up: \Si\times (-\de, \de)\to M: (y, \si) \mapsto \exp_y(\si\nu_y)$ be the Fermi coordinates around $\Si$ with respect to the unit normal $\nu$. For a function $w\in C^\infty(\Si)$, denote the graph $\{\exp_y\big(w(y)\nu_y\big): y\in \Si\}$ of $w$ in Fermi coordinates by $\fgr(w)$. We also let $\th(w)$ simply denote the expansion of $\fgr(w)$ in the unit normal $\pa_\si ^\perp / |\pa_\si^\perp|$ where $\pa_\si^\perp$ is the projection of $\pa_\si$ onto the normal space of $\fgr(w)$. Observe that the operator 
\begin{align*}
    \cT: C^\infty(\Si) \times \R\rightarrow C^\infty(\Si)
\end{align*}
defined by 
\begin{align*}
    \cT(w, \tau) = \th(w) - \tau
\end{align*}
is a Frechet smooth mapping and $\cT(0, \tau_0) = 0$. The linearization of $\cT$ with respect to the first argument at $(0,\tau_0)$ is given by 
\begin{align*}
    (D_1\cT)|_{(0,\tau_0)}(w') = \cL_{\Si} w'
\end{align*}
for $w'\in C^\infty(\Si)$. Since $\la_1(\cL_\Si)>0$, the linearization operator $D_1\cT(0,\tau_0)$ is an isomophism from $C^{\infty}(\Si)$ onto $C^{\infty}(\Si)$. By implicit function theorem, there exists $\vare>0$ and a unique Frechet smooth mapping 
\begin{equation}
    \cS: (\tau_0 -\vare, \tau_0 + \vare) \longrightarrow C^\infty(\Si)
\end{equation}
such that 
\begin{equation*}
    \cS(\tau_0) = 0
\end{equation*}
and for $\tau\in (\tau_0 -\vare, \tau_0 + \vare)$
\begin{equation}
    \cT(\cS(\tau),\tau) = 0.
\end{equation}
Define the smooth one-parameter family of embeddings 
\begin{equation*}
    \Psi:(\tau_0 -\vare, \tau_0 + \vare)\times \Si\longrightarrow M
\end{equation*}
by 
\begin{equation*}
    \Psi(\tau, y) = \exp_y\big(\cS(\tau)(y)\nu_y\big)
\end{equation*}
for $\tau\in (\tau_0 -\vare, \tau_0 + \vare)$, $y\in \Si_0$. Denote the sheet $\Psi(\tau, \Si)$ by $\Si_\tau$. It follows that $\{\Si_\tau : \tau \in (\tau_0 -\vare, \tau_0 + \vare)\}$ is a smooth one-parameter family of closed smooth embedded surfaces with constant expansion $\tau$. Thus, (1) and (2) has been established. The local uniqueness property (3) follows from the contraction principle in the proof of implicit function theorem.

It remains to show that $\Phi$ is a foliation. Observe that $\Psi(\tau, \cdot)$ satisfies the evolution equation
\begin{align}\label{eq: flow}
    \frac{d}{d\tau}\Psi = \psi_{\tau}\nu_{\Si_\tau}
\end{align}
where $\psi_\tau\in C^\infty(\Si_\tau)$ satisfies 
\begin{equation}\label{L phi=1}
    \cL_{\Si_\tau} \psi_\tau = 1.
\end{equation}
To see that $\Psi$ is a foliation, it suffices to show the velocity function $\psi_{\tau}>0$ for all $\tau \in (T_-, T_+)$. Toward contradiction, suppose $\psi_{\tau}(x) \leq 0 $ for some $\tau \in (T_-, T_+)$ and $x\in \Si_\tau$. Let $\beta_\tau>0$ denote the (unique up to scaling) eigenfunction of $\cL_{\Si_\tau}$ associated with the principal eigenvalue $\la_1(\cL_{\Si_\tau})$. There exists $b_\tau\ge 0$ such that $\min_{\Si_\tau} (\psi_{\tau} + b_\tau\beta_\tau) = 0$. At minimum point, by (\ref{L phi=1}) we obtain
\begin{align*}
    0\geq -\De_{\Si_\tau} (\psi_\tau +  b_\tau\beta_\tau) = \cL_{\Si_\tau} (\psi_\tau + b_\tau\beta_\tau) = 1+ b_\tau\la_1(\cL_{\Si_\tau}) \beta_\tau \ge 1.
\end{align*}
This is a contradiction. 
\end{proof}

\begin{cor}[Maximal smooth stable foliation]\label{cor: maximal smooth stable foliation}
Suppose $(\Si, \nu)$ is a closed smooth embedded strictly stable CES in $(M,g,k)$ with $\th_\Si \equiv  \tau_0$ in $(M,g,k)$. Then there exists an open interval $(T_-, T_+)$ containing $\tau_0$ and a smooth CES foliation $\Psi: (T_-, T_+)\times \Si \rightarrow M$ satisfying the following properties:
\begin{enumerate}
    \item[(1)] $\Psi(\tau_0, \cdot) = \mathrm{Id}_\Si(\cdot)$ on $\Si$.
    \item[(2)] $\th_{\Si_\tau} \equiv \tau$ for all $\tau\in (T_-, T_+)$.
    \item[(3)] $\la_1(\cL_{\Si_\tau}) > 0$ for $\tau\in (T_- , T_+)$. 
\end{enumerate}
Furthermore, if $|T_+| < \infty$ (resp. $|T_-| < \infty$), then $\Si_\tau$ converges to a smooth marginally stable CES $\Si_{T_+}$ (resp. $\Si_{T_-}$) as $\tau\to T_+$ (resp. $\tau\to T_-$).
\end{cor}

\begin{proof}
It is known that the principal eigenvalue depends (Lipschitz) continuously on the coefficients of the elliptic operator (cf. \cite{BR}). By the local existence Proposition \ref{prop: local stable foliation} and local estimate Proposition \ref{prop: Local est for foliation of CES}, $\Psi$ can be extended uniquely to an open neighborhood of the slice $\Si_\tau$ as long as $\Si_\tau$ has finite constant expansion and $\la_1(\cL_{\Si_\tau})>0$. Thus, there is a maximal interval $(T_-, T_+)$ such that $\Psi$ remains smooth and satisfies $\th_{\Si_\tau} \equiv \tau$ and $\la_1(\cL_{\Si_\tau}) > 0$ for all $\tau \in (T_-, T_+)$. In particular, if $|T_+| < \infty$ (resp. $|T_-| < \infty$), then $\Si_\tau$ converges smoothly to a CES $\Si_{T_+}$ (resp. $\Si_{T_-}$) as $\tau\to T_+$ (resp. $\tau\to T_-$). In either case, $\Si_{T_+}$ or $\Si_{T_-}$ is marginally stable; otherwise, the foliation $\Psi$ continues by the local construction, which contradicts to the maximality of the interval $(T_-, T_+)$.
\end{proof}

\begin{prop}[Local smooth solution]\label{prop:smooth solution}
Suppose $(\Si, \nu)$ is a closed smooth strictly stable CES with $\th \equiv  \tau_0$ in $(M,g,k)$. Let $\Psi$ be the maximal stable foliation constructed in Corollary \ref{cor: maximal smooth stable foliation}. Define 
\begin{equation}
    v\big(\Psi(\tau, y)\big) = \tau
\end{equation}for all $\tau\in (T_-, T_+)$, $y\in \Si$. Then $v$ is a smooth solution to equation (\ref{Eq:ConstExpLevel}) in the region $\Psi\big((T_-, T_+)\times \Si\big)$ such that $\D v$ is nowhere vanishing. Moreover, for all $\tau \in (T_-, T_+)$ there exists $0 < C(\tau)< \infty$ depending continuously on local geometry of $\Si_\tau$ and $k$ such that 
\begin{equation}\label{Harnack inequality 2}
    C(\tau)^{-1}\la_1(\mathcal{L}_{\Si_\tau}) \leq |\D v|_{\Si_\tau} \leq C(\tau) \la_1(\mathcal{L}_{\Si_\tau}).
\end{equation}
In particular, if $|T_+| < \infty$ (respectively $|T_-| < \infty$), then $\D v(x)$ converges to zero uniformly as $x$ on approach to $\Si_{T_+}$ (respectively $\Si_{T_-}$). 
\end{prop}

\begin{proof}
By definition, $v$ is a smooth function since $\Psi$ is a smooth foliation. Let $\tau \in (T_-, T_+)$. In view of (\ref{eq: flow}) and (\ref{L phi=1}) we have
\begin{equation*}
    1 =  \frac{d}{d \tau}v = \lan \D v, \psi_\tau \nu_{\Si_{\tau}} \ran = |\D v|\cdot\psi_\tau \quad \text{on $\Si_\tau$.}
\end{equation*}
From the proof of Proposition \ref{prop: local stable foliation}, we find $0 < \psi_\tau < \infty$. Thus, 
\begin{equation}\label{velocity Dv}
    0 < |\D v| = 1/\psi_\tau < \infty \quad \text{on $\Si_\tau$.}
\end{equation}  
It follows that the level set $v^{-1}(\tau) = \Si_\tau$ is regular and has constant expansion $\tau$. Therefore, $v$ is a classical solution to (\ref{Eq:ConstExpLevel2}) in $\Psi\big((T_-, T_+)\times \Si\big)$.

Let $\beta_\tau>0$ denote the (unique up to scaling) eigenfunction of $\cL_{\Si_\tau}$ associated with the principal eigenvalue $\la_1(\cL_{\Si_\tau})$. Remark that the following argument is independent of the choice of scaling of $\beta_\tau$. By Harnack inequality, there exists $C(\tau)$ such that 
\begin{equation}\label{Harnack inequality}
    \max_{\Si_\tau} \beta_\tau \leq C(\tau) \min_{\Si_\tau} \beta_\tau
\end{equation}
for all $T_- < \tau < T_+$. Here $C(\tau)$ depends on the coefficients of $\cL_{\Si_\tau}$ and intrinsic diameter of $\Si_\tau$ and therefore depends on local geometry of $\Si_\tau$ and $k$. Since both $\psi_\tau$ and $\beta_\tau$ are positive and $\Si_\tau$ is compact, there exists a constant $b_\tau> 0$ and a point $x_\tau\in \Si_\tau$ such that 
\begin{equation*}
    \max_{\Si_\tau} (\psi_\tau - b_\tau\beta_\tau) = \psi_\tau(x_\tau) - b_\tau\beta_\tau(x_\tau) = 0.
\end{equation*} 
It follows from (\ref{L phi=1}) that 
\begin{align*}
    0\leq \cL_{\Si_\tau}(\psi_\tau - b_\tau\beta_\tau)(x_\tau) = 1-b_\tau\la_1(\cL_{\Si_\tau})\beta_\tau(x_\tau).
\end{align*}
Thus, 
\begin{equation*}
    b_\tau\beta_\tau(x_\tau)\leq \la_1(\cL_{\Si_\tau}).
\end{equation*} 
Then the maximum of $\psi_\tau - b_\tau \beta_\tau$ at $x_\tau$ and the Harnack inequality (\ref{Harnack inequality}) imply that for any $x\in \Si_\tau$
\begin{equation}\label{cla1 upper}
    \psi_\tau(x) \leq b_\tau\beta(x) \leq b_\tau C(\tau)\beta(x_\tau) \leq C(\tau)\la_1(\cL_{\Si_\tau}).
\end{equation}
By considering $\min_{\Si_\tau} (\psi_\tau - a_\tau \beta_\tau) =  0$ for suitable constant $a_\tau >0$, we can analogously show that
\begin{equation}\label{cla1 lower}
    C(\tau)^{-1}\la_1(\cL_{\Si_\tau})\leq \psi_\tau.
\end{equation}
Putting (\ref{velocity Dv}), (\ref{cla1 upper}) and (\ref{cla1 lower}) together, we conclude (\ref{Harnack inequality 2}).

If $|T_\pm| < \infty$, then by Corollary \ref{cor: maximal smooth stable foliation} $\Si_\tau$ converges smoothly to $\Si_{T_\pm}$ smoothly as $\tau \to T_\pm$ and therefore $C(\tau)$ can extend continuously to $\tau = T_\pm$. In view of (\ref{Harnack inequality 2}) and Corollary \ref{cor: maximal smooth stable foliation}: $\la_1(\cL_{\Si_{T_\pm}}) = 0$, we find that $|\D v|(x)$ converges to 0 uniformly as $x$ goes to $\Si_{T_\pm}$.
\end{proof}

When $(\Si, \nu)$ is a compact smooth marginally stable CES, we are not able the construct a local foliation of CES around $\Si$ with the operator in Proposition \ref{prop: local stable foliation}. Nevertheless, Galloway \cite{Ga} constructed a local foliation of CES around $\Si$ by considering the operator
\begin{equation*}
    \cT_0: C^\infty(\Si)\times \R\rightarrow C^\infty(\Si) \times \R, \quad \cT_0(w, \ell) =  \Big(\theta_{\fgr(w)} - \ell, \int_\Si w \Big).
\end{equation*}
The fact that the principal eigenvalue is simple allows him to apply the inverse function theorem with this operator. The drawbacks of the foliation are that the expansion function $\theta$ is implicit and that the sheets are not necessarily stable so that we can further extend the foliation.

\begin{prop}[cf. \cite{Ga} the proof of Theorem 3.1]\label{marginally stable foliation}
Suppose $(\Si, \nu)$ is a compact smooth marginally stable CES with $\th \equiv \tau_0$ and $\la_1(\cL_\Si) = 0$ in $(M,g,k)$. Then there exists $\vare > 0$ and a smooth CES foliation $\Psi_0: (-\vare, \vare)\to M$ satisfying the following properties. Denote $\Psi_0(\tau, \Si)$ by $\Si_\tau$. We have
\begin{enumerate}
    \item[(1)] $\Psi_0(0, \cdot) = \mathrm{Id}_\Si(\cdot)$ on $\Si$.
    \item[(2)] If $\tilde{\Si}$ is a compact smooth CES in $\Psi\big((-\vare, \vare )\times\Si\big)$ and $\tilde{\Si}$ can be expressed as a graph of $w\in C^\infty(\Si)$ in Fermi coordinates around $\Si$, then $\tilde{\Si} = \Si_{\tilde{\tau}}$ where
    \begin{equation}
        \tilde{\tau} = \int_\Si w.
    \end{equation}
\end{enumerate}
\end{prop}

\subsection{Comparison Theorem}

\begin{thm}\label{thm: comparison theorem}
Suppose $\Si_0\subset \pa \Om$ is a compact smooth strictly stable MOTS. Let $u$ be a capillary blowdown limit and let $v$ be the local smooth solution constructed on the annular region $\Psi\big([0, T_+], \Si_0\big)$ in Proposition \ref{prop:smooth solution}. Then $u\leq v$ in $\Psi\big([0, T_+], \Si_0\big)$.
\end{thm}

\begin{proof}
Let $\cA$ denote the annular region $\Psi\big((0, T_+), \Si_0\big)$. Suppose the statement is not true, that is, $u-v>0$ at some point in $\cA$. Since $u-v$ is continuous and $\bar{\cA}$ is compact, there exists $x_0\in \bar{\cA}$ such that $(u-v)(x_0) = \max_{\bar{\cA}} (u-v) > 0$. 

Suppose $x_0 \in \cA$ is an interior point. Let $\rho(x) = d(x, x_0)$. Note that $u-v - \rho^4$ has a strict maximum in $\cA$ at $x_0$. Since $u_{s_j}$ converges to $u$ uniformly on $\bar{\cA}$, there exists $x_j\in \bar{\cA}$ such that $u_{s_j}-v-\rho^4$ has a local maximum at $x_j$ and $x_j\to x_0$. For all sufficiently large $j$, we have $x_j\in \cA$. Since $u_{s_j}$ and $v$ are twice differentiable, the derivative test at $x_j$ shows that
\begin{align*}
    \D u_{s_j}(x_j) = \D v(x_j) + \D\rho^4(x_j), \quad \D^2 u_{s_j}(x_j) \leq \D^2 v(x_j) + \D^2\rho^4(x_j).
\end{align*}
In view of regularized equation (\ref{regeq2}), at $x_j$ we have
\begin{align*}
    0 &= - \D^2u_{s_j}\big(I - \frac{\D u_{s_j}}{\sqrt{s_j^2+ |\D u_{s_j}|^2}}\otimes \frac{\D u_{s_j}}{\sqrt{s_j^2+ |\D u_{s_j}|^2}}\big) \\
    &\quad + \sqrt{s^2 + |\D u_{s_j}|^2}\Big\{u_{s_j} + k\big(I - \frac{\D u_{s_j}}{\sqrt{s_j^2+ |\D u_{s_j}|^2}}\otimes \frac{\D u_{s_j}}{\sqrt{s_j^2+ |\D u_{s_j}|^2}}\big)\Big\}\\
    &\geq - \D^2 v_j\big(I - \frac{\D u_{s_j}}{\sqrt{s_j^2+ |\D u_{s_j}|^2}} \otimes \frac{\D u_{s_j}}{\sqrt{s_j^2+ |\D u_{s_j}|^2}}\big) \\
    &\quad - \D^2\rho^4\big(I - \frac{\D u_{s_j}}{\sqrt{s_j^2+ |\D u_{s_j}|^2}}\otimes \frac{\D u_{s_j}}{\sqrt{s_j^2+ |\D u_{s_j}|^2}}\big) \\
    &\quad + \sqrt{s^2 + |\D u_{s_j}|^2}\Big\{u_{s_j} + k\big(I - \frac{\D u_{s_j}}{\sqrt{s_j^2+ |\D u_{s_j}|^2}}\otimes \frac{\D u_{s_j}}{\sqrt{s_j^2+ |\D u_{s_j}|^2}}\big)\Big\}.\\
\end{align*}
We remark that 
\begin{align*}
    &\lim_{j\to \infty} \D\rho^4(x_j) = \D\rho^4(x_0) = 0,\\
    &\lim_{j\to \infty} \D^2\rho^4(x_j) = \D^2\rho^4(x_0) =  0,\\
    &\lim_{j\to \infty} u_{s_j}(x_j) = u(x_0).
\end{align*}
Moreover, by Proposition \ref{prop:smooth solution} we know $v(x_0)\neq 0$ for $x_0\in \cA$ and hence 
\begin{align*}
    \lim_{j\to \infty} \frac{\D u_{s_j}(x_j)}{\sqrt{s_j^2+ |\D u_{s_j}(x_j)|^2}} = \frac{\D v(x_0)}{|\D v(x_0)|}.
\end{align*}
Therefore, by letting $j\to \infty$ we obtain at $x_0$
\begin{align*}
    0&\geq -\D^2 v(x_0)\big(I - \frac{\D v(x_0)}{|\D v(x_0)|}\otimes \frac{\D v(x_0)}{|\D v(x_0)|}\big) \\
    &\quad + |\D v(x_0)|\{u(x_0) + k\big(I - \frac{\D v(x_0)}{|\D v(x_0)|}\otimes \frac{\D v(x_0)}{|\D v(x_0)|}\big)\}\\
    &= |\D v|\{u(x_0) - v(x_0)\}
\end{align*}
where the last equality follows from the fact that $v$ satisfies equation (\ref{Eq:ConstExpLevel2}). We then conclude that
\begin{align*}
    0< |\D v(x_0)|\{u(x_0) - v(x_0)\} \leq 0
\end{align*}
which is a contradiction.

Next suppose $x_0\in \pa \cA$. Note that $u = v = 0$ on $\pa \Om$, so $x_0 \in v^{-1}(T_+)$. We claim that $\cA\subset E_{u(x_0)}^-(u)$ and therefore $x_0\in \pa E_{u(x_0)}^-(u)$. To confirm this, we observe that if $x\in \cA$, then $v(x) < T_+ $ and by maximality of $u-v$
\begin{align*}
    u(x) = u(x) - v(x) + v(x) \leq u(x_0) - v(x_0) + v(x) < u(x_0).
\end{align*}
 By Theorem \ref{cor:CES in level set}, there exists a closed smooth properly embedded surface $\Si'$ in $u^{-1}(u(x_0))$ passing through $x_0$ having constant expansion $u(x_0)$ with respect to the unit normal vector pointing outside of $E_{u(x_0)}^-(u)$. In addition, the above clam implies that $\Si'\subset u^{-1}(u(x_0))\subset \Om\bsls \cA$. Therefore, $\Si'$ is enclosed by $\Si_{T_+}$ and two surfaces contact each other at $x_0$. Since the chosen unit normal vectors $\nu_{\Si_{T_+}}$ and $\nu_{\Si'}$ of these two surfaces agree at $x_0$ and $\Si'$ lies on the $+\nu_{\Si_{T_+}}$-side of $\Si_{T_+}$, the maximum principle implies that $\rH_{\Si}(x_0)\leq \rH_{\Si_{T_+}}(x_0)$. Consequently, we conclude that 
\begin{align*}
    u(x_0) = \th_{\Si'}(x_0) \leq \th_{\Si_{T_+}}(x_0) = v(x_0)
\end{align*}
which contradicts the assumption that $u(x_0) > v(x_0)$.
\end{proof}

\section{Structure of blowup region and capillary blowdown limit}
\subsection{Thin maximal domains}
For any open subset $S$ of $M$, we can define the \textbf{thickness} of $S$ by
\begin{equation*}
    \tau(S) = \sup \{\mathrm{diam} B: \mbox{$B$ is an open geodesic ball in $S$}\}.
\end{equation*}

\begin{prop}\label{prop: thin domain}
There exists a constant $R_0 = R_0(M,g,k)>0$ satisfying the following property. Let $\Th\in [-\mu_1, \mu_1]$ and let $f$ be a smooth solution to constant expansion equation $\rH(f) -\rK(f) = \Th$ on maximal domain $U$ in $(M,g,k)$. If $U$ is \textbf{thin} in the sense that $\tau(U)< R$, then $f$ has no critical point. Moreover, $U$ is homeomorphic to $f^{-1}(0)\times \R$ with exactly two boundary components $\pa _- U = "f^{-1}(-\infty)"$ and $\pa_+ U = "f^{-1}(\infty)"$ which are closed smooth CES with expansion $\Th$.
\end{prop}

\begin{proof}
Suppose $x_0\in U$ is a critical point of $f$. For $x\in U$, we define
\begin{equation*}
\beta(x):= \lan \nu_f, \pa t\ran|_{(x, f(x))} = (1+ |\D f(x)|^2)^{-1/2}.
\end{equation*}
Thus, $\beta(x_0) = 1$. By Harnack-type inequality in Proposition \ref{Prop: LocalEst} (3),
\begin{align*}
c_4\geq |d\log\beta|_{\tilde{g}}^2 = |d\log\beta|_g^2 - \frac{\lan df, d\log\beta\ran_g^2}{1+ |\D f|_g^2}\geq \beta^2|d\log\beta|_g^2 = |d\beta|_g^2,
\end{align*}
where $\tilde{g} = g + df\otimes df$ is the induced metric on the graph of $f$. Let $\ga: [0, \ell]\rightarrow U$ be a geodesic emitting from $\ga(0) = x_0$. Then for $s\in [0, \ell]$,
\begin{align*}
|\beta(\ga(s)) - \beta(\ga(0))| &= \big|\int_0^s \frac{d}{d\tau} \beta(\ga(\tau))\,d\tau\big| \\
&\leq \int_0^s \big|\lan \D\beta(\ga(\tau)), \ga'(\tau)\ran_g\big| \,d\tau\\
&\leq c_4 \int_0^s d\tau =  c_4s.
\end{align*}
If $\ell< c_4^{-1}$, then $|\beta(\ga(s))-1|< 1$ for any $s\in [0,\ell]$. It follows that $\beta(x)>0$ for any point $x$ in the closed geodesic ball $\bar{B}_\ell(x_0)$, and hence $\D f$ is bounded in the geodesic ball $\bar{B}_\ell(x_0)$. The domain $U$ is maximal, so $U$ contains the closed geodesic ball $\bar{B}_\ell(x_0)$ for any $\ell< c_4^{-1}$. Therefore, $\tau(U)\geq 2 c_4^{-1}$. Take $R_0 =  2 c_4^{-1}$, then we will get a contradiction.

The second statement follows immediately from Morse theory and Theorem \ref{thm: local convergence}.
\end{proof}

\begin{prop}\label{prop: unstable CES in thin maximal domain}
Let $f, U$ with $\tau(U)< R_0$ be assumed as in in Proposition \ref{prop: thin domain}. Let $\nu$ and $\nu'$ be unit normal vector field on $\pa_- U$ and $\pa_+ U$ respectively chosen as in Theorem \ref{thm: local convergence}. Suppose $\pa_- U$ and $\pa_+ U$ are stable (CES). There exists $R>0$ depending on the geometry of $\pa U$ in $(M,g)$ and $k$ such that if $\tau(U)\leq R$, then there exist closed smooth marginally stable CES's $\Si_1$ on the $+\nu$-side of $\pa_-U$ and $\Si_2$ on the $-\nu'$-side of $\pa_+ U$ such that $\Si_i\cap U \neq \emptyset$ for both $i = 1,2$.
\end{prop}
\begin{proof}
By Proposition \ref{prop: thin domain}, $\pa U$ has exactly two component $\pa _- U = "f^{-1}(-\infty)"$ and $\pa_+ U = "f^{-1}(\infty)"$ with the same constant expansion $\Th$. We simply denote $\pa_- U$ by $\Si$ and denote $\pa_+ U$ by $\Si'$. Let $\nu$ and $\nu'$ be unit normal vector field on $\pa_- U$ and $\pa_+ U$ respectively chosen as in Theorem \ref{thm: local convergence}. Thus, $\pa_+ U$ is on the $+\nu$-side of $\pa_- U$. There exists $\rho_0$ depending on the geometry of $\Si$ in $(M,g)$ such that the Fermi coordinates $\Up: \Si \times (-\rho_0, \rho_0) \to M: (x, \si) \to \exp_x ( \si \nu(x) )$ is bijective. Let $w\in C^{2, \al}(\Si)$ with $\|w\|_0 < \rho_0$. Let $\fgr(w)$ denote the graph of $w$ in Fermi coordinates around $\Si$ and let $\th(w)$ denote the expansion of $\fgr(w)$ with respect to $\pa_\si^\perp / |\pa_\si^\perp|$ where $\pa_\si^\perp$ is the projection of $\pa_\si$ onto the normal space of $\fgr(w)$. By the nature of linearization operator $\cL_\Si$, we define the deviation of $\big(\th(w)$ from its linear approximation around $\Si$:
\begin{equation}\label{eq: Lw}
	Q(w) := \big(\th(w) - \Th\big) - \cL_\Si w.
\end{equation}
The dominant (quadratic) part of $Q$ depends also on $Dw$ and $\D^2 w$. Here the notation $Q(w)$ is treated as a functional to simplify the notation. There exist constants $R_0, A$ depending on geometry of $\Si$ in $(M,g)$ and $k$ such that $0< \rho_1 < \rho_0$ and for $\|w\|_{2, \al}\leq \rho_1$ 
\begin{equation}\label{bound for Q(w) 1}
	\|Q(w)\|_{0, \alpha} \leq A \|w\|_{2, \al}^2.
\end{equation}
The standard Schauder estimates applied to (\ref{eq: Lw}) implies that there exists a constant $C$ depending on geometry of $\Si$ in $(M,g)$ and $k$ such that 
\begin{equation}\label{Schauder 1}
	\|w\|_{2,\al} \leq C \big(\|w\|_{0} + \|\th(w) - \Th\|_{0,\al} + \|Q(w)\|_{0,\al} \big).
\end{equation}
Put (\ref{bound for Q(w) 1}) into (\ref{Schauder 1}), we obtain
\begin{align*}
	\|w\|_{2,\al} \leq C \big(\|w\|_{0} + \|\th(w) - \Th\|_{0,\al} \big) + AC\|w\|_{2,\al}^2.
\end{align*}
Let $\de \in (0,1)$. If $\|w\|_{2, \al} \leq \de A^{-1}C^{-1}$, then we have estimates
\begin{align}\label{Schauder 2}
	\|w\|_{2,\al} \leq  (1 - \de)^{-1} C \big(\|w\|_{0} + \|\th(w) - \Th\|_{0,\al} \big),
\end{align}
and
\begin{align}\label{bound for Q(w) 2}
	\|Q(w)\|_{0,\al} \leq \eta\big(\|w\|_{0} + \|\th(w) - \Th\|_{0,\al} \big)
\end{align}
where  $\eta :=  \frac{\de C}{1-\de}\in (0,1)$ if $\de$ is chosen small enough. Take $0< \rho_2< \rho_1$ such that 
\begin{equation}
	\rho_2 < \de(1-\de)A^{-1}C^{-2}.
\end{equation}
Then (\ref{Schauder 2}) and (\ref{bound for Q(w) 2}) hold true as long as $\|w\|_{0} + \|\th(w) - \Th\|_{0,\al} \leq \rho_2$ by using continuity argument along the family $\{sw\}_{0\leq s\leq 1}$. In particular, suppose $\Si'$ can be expressed as the graph of $v>0$ with $\|v\|_0 \leq \rho_2$, then $\th(v) = \Th$  implies
\begin{equation}\label{eq: Lv}
	\cL_\Si v = -Q(v).
\end{equation}
In this case, (\ref{Schauder 2}) and (\ref{bound for Q(w) 2}) can reduced to 
\begin{equation}
	\|v\|_{2,\al} \leq  (1 - \de)^{-1} C \|v\|_{0},
\end{equation}
and
\begin{align}\label{bound for Q(v)}
	\|Q(v)\|_{0,\al} \leq \eta\|v\|_{0}.
\end{align}

Now since $\Si$ is strictly stable, by Corollary \ref{cor: maximal smooth stable foliation} there exists a maximal foliation $\Psi$ of CES initiated from $\Si$ toward the $+\nu$-side. Let $\tau$ be a small positive number, and let $w_\tau>0$ represent the sheet $\Psi(\Th + \tau, \Si)$ in the maximal foliation $\Psi$ satisfying $\th(w_\tau) = \Th + \tau$. Then $w_\tau$ satisfies
\begin{equation}\label{eq: Lw_sigma}
	\cL_\Si w_\tau =  \tau - Q(w_\tau). 
\end{equation}
Since $w_\tau$ and $v$ are both positive, there exists a number $a>0$ and a point $z\in \Si$ such that
\begin{equation}\label{ineq: 5.11}
	av\leq w_\tau \quad \mbox{and the equality holds at $z$}.
\end{equation}
By derivative tests, 
\begin{align*}
	0&\geq \cL_\Si(w - av)(z) = \tau - Q(w_\tau)(z) + a Q(v)(z)\\
	&\geq \tau - \eta(\|w_\tau\|_0 + \tau) - a \eta \|v\|_0\quad \mbox{by (\ref{bound for Q(w) 2}) and (\ref{bound for Q(v)})}\\
	&\geq (1-\eta) \tau - 2\eta \|w_\tau\|_0 \quad \mbox{by (\ref{ineq: 5.11})}
\end{align*}
This implies that 
\begin{equation}\label{bound for sigma}
	\tau \leq 2\eta(1-\eta)^{-1}\|w_\tau\|_0.
\end{equation}
Again by continuity argument, (\ref{bound for sigma}) holds true so long as $\|w_\tau\|_0 \leq [1+2\eta(1-\eta)^{-1}]^{-1}\rho_2:= \rho_3$. In this case, the Schauder estimates (\ref{Schauder 2}) can be further reduced to
\begin{equation}\label{Schauder 3}
	\|w_\tau\|_{2,\al} \leq (1 - \de)^{-1}[1+2\eta(1-\eta)^{-1}]C \cdot \|w_\tau\|_0.
\end{equation}
Combining (\ref{bound for Q(w) 2}) and (\ref{bound for sigma}), we have
\begin{equation}
	\|\cL_\Si w_\tau\|_{0} \leq 3\eta\|w\|_0.
\end{equation}
Then the Harnack inequality applied to (\ref{eq: Lw_sigma}) implies that there exists a constant $\La$ depending on geometry of $\Si$ in $(M,g)$ and $k$ such that 
\begin{equation}
	\|w_\tau\|_0 \leq \La \big( \min w_\tau + \|\cL_\Si w_\tau\|_0 \big)\leq \La \min w_\tau + 3\eta\La \|w_\tau\|_0.
\end{equation}
If $\de$ is chosen small enough (and so is $\eta$) such that $3\eta\La < 1$, then
\begin{equation}\label{lower range of foliation}
	\min w_\tau \geq (1-3\eta\La)\La^{-1}\|w_\tau\|_0.
\end{equation}
Set $\rho_4:= \frac12(1-3\eta\La)\La^{-1}\rho_3$. From (\ref{Schauder 3}), we find the sheet $\Psi(\Th+ \tau, \Si)$ is $C^{2,\al}$ if $\|w_\tau\|_0 \leq \rho_3$.  If the sheets remain stable, then the foliation $\Psi$ would continue and by (\ref{lower range of foliation}) sweep the region $[0, \rho_4]\times \Si$ in Fermi coordinates. On the other hand, the Harnack inequality applied to (\ref{eq: Lv}) implies that
\begin{equation}
	\|v\|_0 \leq \La \big( \min v + \|\cL_\Si v\|_0 \big)\leq \La \min v + \eta\La \|v\|_0.
\end{equation}
Thus, 
\begin{equation}\label{upper bound of the v}
	\|v\|_0 \leq (1-\eta\La)^{-1} \La \min v.
\end{equation}
We set $R: = (1-\eta\La) \La^{-1}\rho_4$.  If $\tau(U)\leq R$, then $\min v\leq \tau(U)$ and (\ref{upper bound of the v}) implies that $\Si' \subset (0,\rho_4]\times \Si$. It follows that there exists a sheet $\Si_\tau := \Psi(\Th+ \tau, \Si)$ in the foliation $\Psi$ for some positive number $\tau$ such that $\Si_\tau$ lies on the $+\nu'$-side of $\Si'$ and contacts $\Si'$ at a point, say $p$. By maximal principle,
\begin{equation}
	\rH_{\Si_\tau}(p) \geq \rH_{\Si'}(p).
\end{equation}
But this contradicts to the fact that $\th_{\Si'} = \Th < \Th + \tau =\th_{\Si_\tau}$. This means that the foliation $\Psi$ towards the $+\nu$-side of $\Si$ must terminate at a marginally stable CES $\Si_1$ which has nonempty intersection with $U$. Analogously, the maximal foliation $\Psi'$ of CES initiated from $\Si'$ towards the $-\nu'$-side must terminate at a marginally stable CES $\Si_2$ which has nonempty intersection with $U$.
\end{proof}

\subsection{Structure theorem}
In the subsection, we will investigate the structure of component $\Om$ of $\Om_+$. We begin with considering $D = \{r_k\}_{k = 1}^\infty\subset  \bar{\Om}$ a dense countable subset. We will apply Theorem \ref{thm: local convergence} multiple times without explicitly mentioning it throughout this subsection. Use diagonal argument and relabeling index $j$, we may assume for all $r_k\in D$ the sequence $\gr(\tilde{f}_{s_j}^{(r_k)})$ converges in $C_{loc}^\infty$ to either a maximal graph or a cylinder over a closed smooth surface. We then decompose $\N$ as $A\sqcup B$ where 
\begin{align*}
    k\in A &: \mbox{$\tilde{f}_{s_{j}}^{(r_k)}$ converges in $C^\infty_{loc}$ to $\tilde{f}_{0}^{(r_k)}$ on maximal domain $U_{r_k}\subset E_{u(r_k)}(u, \Om)$},\\
    k\in B &: \mbox{$\gr(\tilde{f}_{s_{j}}^{(r_k)})$ converges in $C^\infty_{loc}$ to a cylinder over $\Si_{r_k}\subset E_{u(r_k)}(u, \Om)$}.
\end{align*}

\begin{lem}\label{avoidance property}
$\{U_{r_k}\}_{k\in A}$ and $\{\Si_{r_\ell}\}_{\ell\in B}$ satisfy avoidance property. More precisely, 
\begin{enumerate}
    \item[(1)] For $k\in A$ and $\ell\in B$, $U_{r_k}\cap \Si_{r_\ell}=\emptyset$ ; 
    \item[(2)] If $U_{r_k} \cap U_{r_\ell} \neq \emptyset$ for $k,\ell\in A$, then $U_{r_k} = U_{r_\ell}$ ;
    \item[(3)] If $\Si_{r_k}\cap \Si_{r_\ell} \neq \emptyset$ for $k,\ell\in B$, then $\Si_{r_k} = \Si_{r_\ell}$ .
\end{enumerate}
\end{lem}

\begin{proof}
To prove (1), suppose $p\in U_{r_k}\cap \Si_{r_\ell}$. By Theorem \ref{thm: local convergence}, as $p\in U_{r_k}$ 
\begin{align*}
    \lim_{j\to \infty} |\D f_{s_{j}}(p)|  < +\infty.
\end{align*}
and as $p\in \Si_{r_\ell}$
\begin{align*}
    \lim_{j\to \infty} |\D f_{s_{j}}(p)| = +\infty.
\end{align*}
It leads to a contradiction. 

To prove (2), suppose $p\in U_{r_k}\cap U_{r_\ell}$.  By definition, we have the conversion identity
\begin{equation}\label{conversion identity}
    \tilde{f}_{s_j}^{(r_k)}(x) - \tilde{f}_{s_j}^{(r_\ell)}(x) = \tilde{f}_{s_j}^{(r_k)}(r_\ell)
\end{equation}
for all $j, k, \ell\in \N$ and $x\in M$. It follows that by letting $s_j\to 0^+$
\begin{equation*}
    \tilde{f}_0^{(r_k)}(r_\ell) = \tilde{f}_0^{(r_k)}(p) - \tilde{f}_0^{(r_\ell)}(p),
\end{equation*}
and thus
\begin{equation*}
    \tilde{f}_0^{(r_k)}(x) - \tilde{f}_0^{(r_\ell)}(x) = \tilde{f}_0^{(r_k)}(r_\ell)
\end{equation*}
for all $x\in U_{r_k}\cup U_{r_\ell}$. This means that $\tilde{f}_0^{(r_k)}, \tilde{f}_0^{(r_\ell)}$ only differ by a constant. By Theorem \ref{thm: local convergence}, $u \equiv u(p)$ in $U_k\cup U_{r_\ell}$ and $\tilde{f}_0^{(r_k)}, \tilde{f}_0^{(r_\ell)}$ are both solutions to $\rH(f) - \rK(f) = u(p)$. Since $U_{r_k}$ and $U_{r_\ell}$ are maximal domains in the sense that solutions blow up on the boundary, we immediately have $U_{r_k} = U_{r_\ell}$.  

To show (3), suppose $p\in \Si_{r_k}\cap \Si_{r_\ell}$. We first claim that $\Si_{r_k}$ and $\Si_{r_\ell}$ contact at $p$ but can not cross each other. By Theorem \ref{thm: local convergence}, $\lim_{j\to \infty} \frac{D_{s_j}(p)}{\sqrt{1+ |D_{s_j}(p)|^2}}$ is the common unit normal to $\Si_{r_k}$ and $\Si_{r_\ell}$ along which the expansions are both $u(p)$. Suppose $\Si_{r_k}$ crosses $\Si_{r_\ell}$, then there are points $q_\pm$ in $\cN_{\de_{r_k}}^\pm(\Si_{r_k},\nu_{\Si_{r_k}})\cap\cN_{\de_{r_\ell}}^\mp(\Si_{r_\ell},\nu_{\Si_{r_\ell}})$ respectively. It follows from Theorem \ref{thm: local convergence} and the conversion identity (\ref{conversion identity}) at $q_\pm$ that $\lim_{j\to \infty} \tilde{f}_{s_{j}}^{(r_\ell)}(r_k)$ is both $+\infty$ and $-\infty$, which is a contradiction. Therefore, $\Si_{r_k}$ and $\Si_{r_\ell}$ contact each other from one side and both have constant expansion $u(p)$. By strong maximum principle, we find $\Si_{r_k} = \Si_{r_\ell}$. 
\end{proof}

\begin{thm}[Structure Theorem]\label{thm: structure theorem}
Assume that any compact subset of $M$ contains only finitely many marginally stable CES. Let $u$ be a capillary blowndown limit of $f_s$ and let $\Om\subset \Om_+$ be a component of blowup region, say $f_{s_j}\to +\infty$ on $\Om$ and $s_j f_{s_j}\to u$ uniformly on $\Om$.  Then there exists a partition
\begin{align*}
    \Bar{\Om} = (\bigcup_{m = 1}^{N_1} U_m ) \cup (\bigcup_{n = 1}^{N_2} \Phi_n([0, b_n]\times \Si_n))
\end{align*}
where $1\leq N_1, N_2 <\infty$, $U_m$ is a maximal domain of a solution to constant expansion equation $\rH(f)-\rK(f) = u(U_m)$, and $\Phi_n:[0,b_n]\times \Si_n\rightarrow M$ is a smooth foliation of closed CES with $\theta|_{\Phi(\cdot, \Si_n)} = u|_{\Phi(\cdot, \Si_n)}$ with $b_n\geq 0$ (if $b_n = 0$ the foliation degenerates to one sheet of CES). 
\end{thm}

\begin{proof}[Proof of Theorem \ref{thm: structure theorem}]
For simplicity, we identify and then relabel the objects in $\{U_{r_k}\}$ as $\{U_m\}_{m=1}^{N_1}$ such that $U_m \cap U_n = \emptyset$ if $m \neq n$. 

We prove that $1\leq N_1 < \infty$. Theorem \ref{thm:interior of level set} implies that $N_1\geq 1$. Suppose $N_1 = \infty$. By compactness of $\bar{\Om}$, avoidance property of ${U_m}$ and local estimates in Proposition \ref{Prop: LocalEst}, there exists a subsequence $\{U_{m'}\}$ and an accumulation CES $\Si^*$ such that $\tau(U_{m'})<R_0$ as in Proposition \ref{prop: thin domain} and boundary components $\pa_\pm U_{m_k}$ converge to $\Si$ from one side smoothly as $m\to \infty$. Proposition \ref{prop: unstable CES in thin maximal domain} gives the constant $R$ depending only on the geometry of $\Si^*$ in $(M,g)$ and $k$. For large enough $m'$, components $\pa_\pm U_{m'}$ can be written as graphs over $\Si^*$ with very small sup-norm, say less than $R$. This means that $\tau(U_{m'})\leq R$. By the virtue of estimate (\ref{Schauder 3}), we may assume $R$ is also applicable to $\pa_\pm U_{m'}$ for sufficiently large $m'$. By finiteness of the number of marginally stable CES in $\Om$, components of $\pa_\pm U_{m'}$ are strictly stable except finitely many. By Proposition \ref{prop: unstable CES in thin maximal domain}, for every sufficiently large $m'$ there exists a closed smooth marginally stable CES $\tilde{\Si}_{m'}$ which lies between $\Si^*$ and the further boundary component of $U_{m'}$ such that $U_{m'}\cap \tilde{\Si}_{m'}\neq \emptyset$. These CES $\Si_{m'}$ are distinct because of avoidance property of $\{U_{m'}\}$ and $\tilde{\Si}_{m'}$'s relative position to $U_{m'}$ and $\Si^*$. This means that there are infinitely many distinct closed smooth marginally stable CES in $\bar{\Om}$, a contradiction to our assumption. Thus, $N_1<\infty$. As a consequence, we have rather simple topological relations: $\mathrm{Int}(\bar{\Om} - \bigcup_{m = 1}^{N_1} U_{m}) = \Om - \bigcup_{m = 1}^{N_1} \bar{U}_{m}$ and $\pa(\bar{\Om} - \bigcup_{m = 1}^{N_1} U_{m}) = \bigcup_{m = 1}^{N_1} \pa U_{m} \cup \pa \Om$.

By Proposition \ref{prop: stability of CES}, for every $r_k\in \Om - \bigcup_{m = 1}^{N_1} \bar{U}_{m}$ where $k\in B$, $\Si_{r_k}$ is a stable with expansion $u(r_k)$. Then we can use Proposition \ref{prop: local stable foliation} for strictly stable CES or \ref{marginally stable foliation} for marginally stable CES to construct a unique local foliation of CES around $\Si_{r_k}$. Since $D\cap (\Om - \bigcup_{m = 1}^{N_1} \bar{U}_{m})$ is a dense subset in $\Om - \bigcup_{m = 1}^{N_1} \bar{U}_{m}$,  we conclude that each (open) connected component of $\Om - \bigcup_{m = 1}^{N_1} \bar{U}_{m}$ is a foliation of CES. All such foliations can be uniquely and smoothly extended to CES in $\bigcup_{m = 1}^{N_1} \pa U_{m} \cup \pa \Om$ by connectness of $\Om$. There may also exist some isolated CES's which are either the common boundaries of two adjacent maximal domains or components of $\pa \Om$. These isolated CES can be express as degenerate foliations. Therefore, the total number of components of foliations of CES is bounded by the number of components of $\bigcup_{m = 1}^{N_1} \pa U_{m}\cup \pa \Om$.
\end{proof}

\begin{rmk}
\begin{enumerate}
	\item[(1)] Without the assumption of finite marginally stable CES in compact sets, there may be infinitely many disjoint maximal domains. This would causes the complexity of the topology of blowup region and capillary blowdown limit.
	\item[(2)] Although we have $u\in C^{0,1}$ from the construction in Section 3, $u$ is generally not $C^1$.
\end{enumerate}
\end{rmk}

\begin{cor}\label{cor: everywhere convergence subsequence}
Under the assumption of Theorem \ref{thm: structure theorem}, there exists a sequence $s_j\to 0^+$ such that the sequence of graphs of translated functions $\tilde{f}_{s_j}^{(x_0)}$ converges to a smooth submanifold in an open neighborhood of $(x_0, 0)\in M\times\R$ for all $x_0\in \bar{\Om}$. 
\end{cor}

\begin{proof}
We will check that the subsequence $s_j$ obtained by diagonal argument satisfies the claim. Suppose $p\in U_{r_k}$ for $k\in A$, then it follows from the argument in the proof of Lemma \ref{avoidance property} (ii) that $\tilde{f}_{s_j}^{(p)}$ converges to $\tilde{f}_0^{(r_k)} + C$ in $U_{r_k}$ for some constant $C$.

Suppose $p\in \bar{\Om} - \bigcup_{m = 1}^{N_1} U_{m}$. By passing to a further subsequence $s_{j'}$,  $\gr(\tilde{f}_{s_{j'}}^{(p)})$ converges to either a graph on maximal domain $U_p$ or a cylinder over a closed smooth CES $\Si_p$, either of which satisfies avoidance property together with $\{U_{r_k}, \Si_{r_\ell}\}$. Note $(\bigcup_{k\in A} U_{r_k})\cup\,( \bigcup_{\ell\in B} \Si_{r_\ell})$ is a dense subset of $\bar{\Om}$. Thus, the limit submanifold of $\gr(\tilde{f}_{s_{j'}}^{(p)})$ must be a cylinder over a closed smooth CES $\Si_p$ containing $p$. Then we may encounter two scenarios: $p\in \pa (\bar{\Om} - \bigcup_{m = 1}^{N_1} U_{m})$ and $p\in \mathrm{Int}(\bar{\Om} - \bigcup_{m = 1}^{N_1} U_{m})$. In the first scenario, either $p\in \pa U_{m}$ for some $m$ or $p\in \pa \Om$. By avoidance property, $\Si_p$ lies outside of $U_{m}$ but inside of $\bar{\Om}$ and contacts $\pa U_m$ or $\pa \Om$ respectively at a point $p$. It follows from Theorem \ref{thm: local convergence} that at the point $p$ the unit normal vectors of $\Si_p$ and $\pa U_m$ or $\pa \Om$ respectively are identical. Since $\Si_p$ and $\pa U_m$ or $\pa \Om$ respectively both satisfy $\rH-\rK \equiv u(p)$ with a contact point, strong maximum principle implies that $\Si_p$ is a connected component of $\pa U_{m}$ or $\pa Om$ respectively. In the second scenario, $p$ lies in the interior of one foliation of CES $\Phi_n\big( (0, b_n)\times \Si_n)$ in Theorem \ref{thm: structure theorem}. The avoidance property and local uniqueness of foliation around stable CES imply that $\Si_p$ is a sheet of the foliation $\Phi_n$. In both scenarios, we found that $\Si_p$ is uniquely determined. Thus, we can drop the dependence of the choice of subsequence and the convergence holds true for the original sequence $s_j$.
\end{proof}

The pair $(u,\eta)$ of capillary blowdown limit and its companion vector field preserve the geometric information of regularized solutions when the blowup occurs.
\begin{cor}\label{cor: companion vector field}
Assume that any compact subset of $M$ contains only finitely many marginally stable CES. Let $u$ be a capillary blowdown limit of regularized solutions. Then there exists a continuous, piecewise smooth vector field $\eta$ on $M$ satisfying the following properties.
\begin{enumerate}
    \item[(1)] $|\eta(x)| \leq 1$ for all $x\in M$.
    \item[(2)] If $x$ lies in a maximal domain $U$ of a solution $f$ to constant expansion equation in the Structure Theorem \ref{thm: structure theorem} or $\Om_0$ associated with Jang equation, then $\eta$ is the horizontal projection of Gauss map on $\gr(f)$:
\begin{align*}
\eta(x) = \frac{\D f(x)}{\sqrt{1+|\D f(x)|^2}}.
\end{align*}
    \item[(3)] If $x$ lies in a foliation of CES in the Structure Theorem \ref{thm: structure theorem}, then $\eta(x)$ is the unit normal to the CES which contains $x$.
    \item[(4)] The pair $(u,\eta)$ satisfies the equation
    \begin{align}\label{Eq: companion equation}
        \Div_M (\eta) - \tr_g k + k(\eta,\eta) = u \quad \mbox{in $M$}.
    \end{align}
\end{enumerate}
\end{cor}

\begin{proof}
We know from Proposition \ref{prop: SY Jang equation} that both $\Om_+$ and $\Om_-$ have only finitely many connected components (blowup regions). Applying Corollary \ref{cor: everywhere convergence subsequence} to all blowup regions, there exists a decreasing subsequence $s_{j'} \to 0$ such that
\begin{align*}
    \eta(x):= \lim_{j' \to \infty} \frac{\D f_{s_{j'}}(x)}{\sqrt{1+ |\D f_{s_{j'}}(x)|^2}} \quad \text{exists for all $x\in \bar{\Om}_+\cup \bar{\Om}_-$.}
\end{align*}
Proposition \ref{prop: SY Jang equation} implies that the above limit $\eta(x)$ also exists for $x\in \Om_0$ with the same subsequence $s_{j'}$ and 
\begin{equation}
	\eta(x) = \frac{\D f_0(x)}{\sqrt{1+ |\D f_0(x)|^2}}
\end{equation} where $f_0$ is the solution to Jang equation in Proposition \ref{prop: SY Jang equation}. Claim (1) is clear. Claim (2) and claim (3) follow from Proposition \ref{thm: local convergence}. Therefore, $\eta$ is continuous everywhere, and smooth except across boundaries of maximal domains and $\pa \Om_0$. Claim (4) records the fact that the solution $f$ in (2) and CES in (3) satisfy constant expansion equation $\th = u$.
\end{proof}

\begin{cor}[Volume estimate for blowup region] Let 
\begin{equation*}
    I = I(M,g) = \inf \frac{|\pa R|^{\frac32}}{|R|}
\end{equation*}
be the isoperimetric constant of $(M,g)$ where $R$ is any bounded domain whose boundary is nice enough to define area. Suppose $\Om$ is a component of $\Om_-$ or $\Om_+$. Then we have the volume estimate for $\Om$:
\begin{equation}
    |\Om| \geq I^2 \|k\|_{0;\Om}^{-3}
\end{equation}
and the area estimate for $\pa \Om$:
\begin{equation}
    |\pa\Om| \geq I^2 \|k\|_{0;\Om}^{-2}.
\end{equation}
\end{cor}

\begin{proof}
Suppose $\Om\subset \Om_+$ is a connected component. Integrating (\ref{Eq: companion equation}) over $\Om$ and use divergence theorem,
\begin{equation*}
     |\pa \Om| + \int_\Om u = -\int_{\pa \Om}  \lan \eta, \nu \ran +\int_\Om u = -\int_\Om \big\{\tr k - k(\eta, \eta)\big\}.
\end{equation*}
Since $u\geq 0$, by isoperimetric inequality we have
\begin{equation*}
    I^{\frac23}|\Om|^{\frac23} \leq|\pa \Om| \leq |\Om|\|k\|_{0;\Om}.
\end{equation*}
Thus,
\begin{equation*}
     |\Om| \geq I^2 \|k\|_{0;\Om}^{-3}
\end{equation*}
and
\begin{equation}
    |\pa\Om| \geq I^2 \|k\|_{0;\Om}^{-2}.
\end{equation}
If $\Om$ is a connected component of $\Om_-$, then we have
\begin{equation*}
-|\pa \Om| + \int_\Om u = -\int_\Om \big\{\tr k - k(\eta, \eta)\big\}
\end{equation*}
where $u \leq 0$ in $\Om$. Thus, we conclude the same result as for $\Om_+$.
\end{proof}

\section{Trival capillary blowdown limit}
\noindent This section is contributed to the discussion of a very special blowup phenomenon. Given a blowup sequence of regularized solutions, if the speed of caps escaping to infinity is much slower than the contractive rescaling factor $s$, then it is likely that the rescaled sequence ends up with the trivial capillary blowdown limit which is identically zero. There is no obvious evidence excluding this possibility. Nevertheless, trivial capillary blowdown limit is rigid and gives a topological restriction on blowup regions.
\subsection{Rigidity of trivial capillary blowdown limit} In general, uniqueness of capillary blowdown limits on a given blowup region is not clear. Whereas trivial capillary blowdown limit has the following rigidity property. 
\begin{prop}\label{rigidity theorem}
If there exists a sequence $s_j\to 0^+$ such that 
\begin{align*}
    \lim_{j\to \infty} \sup_{x\in M} |u_{s_j}(x)| = 0,
\end{align*}
then
\begin{align*}
    \lim_{s\to 0^+} \sup_{x\in M} |u_{s}(x)| = 0.
\end{align*}
\end{prop}

The proof is based on the monotonicity property of $\max_M |u_s|$. To show this, we need the following gap estimate.
\begin{lem}[Estimate of gap]\label{GapEst}
Suppose $0< t < s$ and suppose $f_s$, $f_t$ are solutions to (\ref{eq: regularized eq}) and converge to 0 at each infinity end. Denote $u_s = sf_s$ and $u_t = tf_t$.
\begin{enumerate}
    \item[(1)] If $\min \{\max_M u_s, \max_M u_t\}>0$, then $\sup_M (f_t-f_s) \leq \frac{s-t}{st}\min \{\max_M u_s, \max_M u_t\}$. 
    \item[(2)] If $\max\{\min_M u_s, \min_M u_t\}<0$, then $\frac{s-t}{st} \max\{\min_M u_s, \min_M u_t\}< \inf_M (f_t-f_s)(x)$. 
\end{enumerate}
\end{lem}
\begin{proof}
We may assume $\sup_M (f_t-f_s)>0$; otherwise, there is nothing to prove. Since $f_t-f_s$ is smooth and decays to zero near infinity, there is $x_0 \in M$ such that $(f_t-f_s)(x_0) = \max_M (f_t-f_s)$. By derivative test, we have
\begin{align*}
    \D (f_t-f_s)(x_0) = 0,\quad \D^2(f_t-f_s)\leq 0.
\end{align*}
By subtracting regularized equations (\ref{eq: regularized eq}) associated with $s$ and $t$, we obtain
\begin{align}\label{ineq: GapEst 1}
    0&\geq \Big(g^{ij}- \frac{f_s^if_s^j}{1+|\D f_s|^2}\Big)\frac{\D_i\D_j (f_t-f_s)}{\sqrt{1+|\D f_s|^2}}(x_0) = tf_t(x_0)- sf_s(x_0).
\end{align}
There are two ways to split the difference. Firstly, we have
\begin{align*}
	0 &\geq  tf_t(x_0)- sf_s(x_0)\\
	&= t f_t(x_0) - t f_s(x_0) + t f_s(x_0) - s f_s(x_0)\\
	&=t \big(f_t(x_0) - f_s(x_0)\big) + (t - s) f_s(x_0).
\end{align*}
Thus,
\begin{align}\label{ineq: GapEst 2}
    f_t(x_0) - f_s(x_0) \leq \frac{s-t}{t}f_s(x_0) \leq \frac{s-t}{t}\max f_s = \frac{s-t}{st}\max u_s.
\end{align}
Secondly, we have
\begin{align*}
	0 &\geq  tf_t(x_0)- sf_s(x_0)\\
	&= t f_t(x_0) - s f_t(x_0) + s f_t(x_0) - s f_s(x_0)\\
	&=(t - s) f_t(x_0) + s \big(f_t(x_0) - f_s(x_0)\big).
\end{align*}
Thus,
\begin{align}\label{ineq: GapEst 3}
    f_t(x_0) - f_s(x_0) \leq \frac{s-t}{s}f_t(x_0)\leq \frac{s-t}{s}\max f_t = \frac{s-t}{st}\max u_t.
\end{align}
Therefore, $f_t(x) - f_s(x) \leq f_t(x_0) - f_s(x_0) \leq \frac{s-t}{st}\min\{\max u_s, \max u_t\}$.

The result (2) follows analogously.
\end{proof}

\begin{cor}\label{Monotonicity of tips}
Suppose $0< t < s$ and suppose $f_s$, $f_t$ are solutions to (\ref{eq: regularized eq}) and converge to 0 at each infinity end. Denote $u_s = sf_s$ and $u_t = tf_t$. Then
\begin{enumerate}
    \item[(1)] If $\max_M u_s > 0$, then $\max_M u_t \leq \max_M u_s$.
    \item[(2)] If $\min_M u_s < 0$, then $\min_M u_s \leq \min_M u_t$.
\end{enumerate}
\end{cor}
\begin{proof}
Suppose $u_t$ achieves its maximum at $\bar{x}$. We may also assume that $\max_M u_t > 0$; otherwise, there is nothing to prove. Then Lemma \ref{GapEst} implies that
\begin{align*}
    \max_M u_t &= tf_t(\bar{x}) = t(f_t(\bar{x})-f_s(\bar{x})) + tf_s(\bar{x})\\
    &\leq \frac{s-t}{s}\max u_s + \frac{t}{s}u_s(\bar{x})\\
    &\leq \max_M u_s.
\end{align*}

The result (2) follows analogously.
\end{proof}

\begin{proof}[Proof of Proposition \ref{rigidity theorem}]
It follows from Corollary \ref{Monotonicity of tips} that $\sup_M |u_s|$ is increasing in $s$. Therefore, $\sup_M |u_s|$ converges to zero as $s \to 0^+$ if one sequence does. 
\end{proof}

\subsection{Topology of blowup region with trivial capillary blowdown limit}
The main theorem of this subsection asserts that when the dominant energy condition holds strictly for the initial data set, if a capillary blowdown limit of $f_s$ is trivial in some blowup region $\Om$, then $\Om$ has rather simple topology.
\begin{thm}\label{thm: topological condition}
Suppose the dominant energy condition holds strictly, i.e. $\mu - |J|_g >0$. Let $u$ be a capillary blowdown limit of $f_s$ and $\Om$ be  a connected component of $\Om_+$ or $\Om_-$ with boundary components $\Si_1, \ldots, \Si_l$. Suppose $u = 0$ in $\Om$. Then the compactification $\Om\cup \{P_1,\ldots, P_l\}$ by adding a point to each boundary component is homeophorphic to a connected sum of finite number of spherical space forms $S^3/\Ga$ and $S^2\times S^1$. 
\end{thm}

We begin with the model case where the entire blowup region $\Om = U$ is one maximal domain of a solution $f$ to Jang equation.
\begin{prop}\label{prop: maximal domain is Yamabe positive}
Let $U$ be a bounded maximal domain of solution $f$ to Jang equation with boundary components $\{\Si_1, \ldots, \Si_l\}$. Suppose the dominant energy condition holds strictly, i.e. $\mu - |J|_g \geq \de$ for some $\de > 0$. Then every boundary component of $U$ is a 2-sphere and the compactification $U\cup \{P_1,\ldots, P_l\}$ by adding a point to each boundary component is homeomorphic to a smooth manifold of positive Yamabe type, i.e. the manifold admits a metric such that the scalar curvature is positive.
\end{prop}

\begin{rmk}\label{rmk: boundary is 2-sphere}
The claim that every boundary component of $U$ is a 2-sphere will follow the same argument of the statement that every boundary component of $\Om_0$ in Proposition \ref{prop: SY Jang equation} is a 2-sphere provided the strict dominant energy condition holds in \cite{SY2}.
\end{rmk}

To construct a compact smooth manifold out of $U$, we need the following gluing lemma to cap off the openings $\pa U$ by topological half 3-spheres. Note that the function $u$ in the conformal factor here no longer represents a blowdown limit. 
\begin{lem}\label{lem: gluing}
Suppose $(\Si, \ga_*)$ is a 2-dimensional compact manifold with or without boundary. Let $\ga_s(x) = e^{2w(x,s)} \ga_*(x)$ for $s\in (a,b)$ be a smooth path in the conformal class of $\ga_*$. Suppose for each $s\in (a,b)$ the first (Neumann if $\Si$ has boundary) eigenvalue of $2$-dimensional conformal Laplacian $\la_1(-\De_{\ga_s}+ \ka_{\ga_s})\geq \la_*$ for some $\la_*>0$ where $\ka_{\ga_s}$ is the Gaussian curvature of $\Si$ with respect to $\ga_s$. Suppose $w$ satisfies the boundary condition
\begin{align}\label{boundary condition}
    \frac{d}{ds}\Big|_{s=a^+} e^{2w(s, x)} = \frac{d}{ds}\Big|_{s=b^-} e^{2w(x,s)} =0 \quad \mbox{for all $x\in \Si$},
\end{align}
and 
\begin{align*}
	\sup_{s\in (a,b), x\in \Si}\, \Big|\frac{d^2}{ds^2} e^{2w}\Big| < \infty.
\end{align*}
Then the first Neumann eigenvalue of 3-dimensional conformal Laplacian $\la_1(-\De_g + \frac{1}{8}\rR_{g})$ on the cylinder $\cC: = \Si\times (a,b)$ equipped with the warped product $g(x,s) = \ga_s(x) + ds^2$ is positive $(\geq \frac{\la_*}{4}>0)$.
\end{lem}
\begin{proof}
Let $i_s: \Si \hookrightarrow \Si\times (a,b)$ denote the inclusion map $i_s(x) = (x, s)$ for $s\in (a,b)$ and $x\in \Si$. Then
\begin{align*}
    \int_{\cC} |d\phi|_g^2 dV_g &= \int_a^b\int_\Si \Big\{|i_s^*d\phi|_{\ga_s}^2 + |\phi'|^2 \Big\} dA_{\ga_s}ds\\
    &= \int_a^b\int_\Si |i_s^*d\phi|_{\ga_s}^2 dA_{\ga_s}ds + \int_a^b\int_\Si |\phi'|^2 e^{2w} dA_{\ga_*}ds
\end{align*}
where $\cdot'$ denotes $\frac{d}{ds}\cdot$. By direct computation, the scalar curvature $\rR_g$ of the warped product metric is
\begin{align*}
    \rR_g = 2\ka_{\ga_s} -4(w'') - 6(w')^2.
\end{align*}
Let $\phi\in C^1(\cC)$ be bounded. Then 
\begin{align*}
    &\quad \frac{1}{8}\int_\cC \rR_g \phi^2 dV_g \\
    &= \frac{1}{4} \int_a^b\int_\Si \ka_{\ga_s}\phi^2 dA_{\ga_s}ds + \int_a^b\int_\Si \Big\{ -\frac{1}{2}w''\phi^2 e^{2w}  - \frac{3}{4} (w')^2 \phi^2 e^{2w}\Big\} dA_{\ga_*}ds\\
    &= \frac{1}{4} \int_a^b\int_\Si \ka_{\ga_s}\phi^2 dA_{\ga_s}ds + \int_a^b\int_\Si \Big\{ w'\phi\phi'e^{2w}  + \frac{1}{4} (w')^2\phi^2 e^{2w}\Big\} dA_{\ga_*}ds.
\end{align*}
In the last equality, we integrate the second term by parts and use the boundary condition (\ref{boundary condition}). Putting above computations together gives
\begin{align*}
    &\quad \int_\cC \Big\{|d\phi|_g^2 + \frac{1}{8} \rR_g \phi^2 \Big\} dV_g \\
    &\geq \frac{1}{4}\int_a^b\int_\Si \Big\{|i_s^*d\phi|_{\ga_s}^2 + \ka_{\ga_s}\phi^2 \Big\} dA_{\ga_s}ds + \int_a^b\int_\Si e^{-\frac{w}{2}}[(\phi e^{\frac{w}{2}})']^2 dA_{\ga_s}ds\\
    &\geq \frac{1}{4} \int_a^b \la_1(-\De_{\ga_s}+ \ka_{\ga_s})\int_\Si \phi^2 dA_{\ga_s}ds\\
    &\geq \frac{\la_*}{4} \int_\cC \phi^2 dV_g.
\end{align*}
Consequently, $\la_1(-\De_g + \frac{1}{8}\rR_g)\geq \frac{\la_*}{4}>0$.
\end{proof}

\begin{rmk}
The boundary condition (\ref{boundary condition}) is weaker than the condition that $\lim_{s\to a^+} \rH_{i_s(\Si)} = \lim_{s\to b^-} \rH_{i_s(\Si)} = 0$ where $H_{i_s(\Si)}$ is the mean curvature of $i_s(\Si)$ in $\cC$ with respect to $\frac{\pa}{\pa s}$. In the following application of this lemma, we will take the cylinder $S^2\times  (0,b)$ to be a flat punctured 3-ball in spherical coordinate near the origin. In this case, $w(x,s) = \log s$ near $s=0$ and the mean curvature of sphere actually blows up near the origin, whereas $\frac{d}{ds} e^{2w}$ converges to zero near the origin.
\end{rmk}

\begin{proof}[Proof of Proposition \ref{prop: maximal domain is Yamabe positive}]
Let $G = \gr(f, U)\subset (M\times \R, g + dt^2)$ endowed with induced metric $\bar{g} = g + df\otimes df$. Observe that vertical translations generate a Jacobi vector field whose normal component is 
\begin{align*}
    \beta:= \lan -\pa t, \nu_G \ran = (1 + |Df|_g^2)^{-\frac{1}{2}}.
\end{align*}
Using identities $\cL_G \beta = 0$ and (\ref{op2}) we find 
\begin{align}\label{Schoen-Yau identity}
    2(\mu -J(\nu)) + |h - k|_{\bar{g}}^2 = -2\Div_G(\xi + \bar{\D}\log \beta ) - 2 |\xi + \bar{\D}\log \beta|_{\bar{g}}^2 + \rR_{\bar{g}}
\end{align}
where $\xi = \big(k(\nu,\cdot)^\sharp\big)^\top$. Choose $t_0>0$ sufficiently large to be determined. Let $\phi\in C^1(G)$. Multiplying (\ref{Schoen-Yau identity}) by $\phi^2$, integrating by parts and using the pointwise Cauchy-Schwartz inequality
\begin{align*}
    2\lan X ,\bar{\D}\phi\ran_{\bar{g}} \phi - |X|_{\bar{g}}^2\phi^2 \leq 2|X|_{\bar{g}}\,|d\phi|_{\bar{g}} |\phi| - |X|_{\bar{g}}^2\phi^2  \leq |d\phi|_{\bar{g}}^2,
\end{align*}
we find
\begin{align}\label{YP: ineq 1}
\begin{split}
    \quad \int_{G\cap \big(M\times (-t_0, t_0)\big)} 2(\mu -J(\nu))\phi^2dV_{\bar{g}} &\leq 2\int_{G\cap (\{\pm t_0\}\times M)} \phi^2 |\lan \xi + \bar{\D}\log \beta, \eta_\pm \ran_{\bar{g}}|dA_{\bar{g}} \\
    &\quad +\int_{G\cap \big(M\times (-t_0, t_0)\big)} 2|d\phi|_{\bar{g}}^2 +  \rR_{\bar{g}}\phi^2dV_{\bar{g}}
\end{split}
\end{align}
where $\eta_\pm = \pm\frac{\D f + |\D f|^2\pa_t}{|\D f|\sqrt{1+ |\D f|^2}}$ is the conormal on the section $G\cap (M\times \{\pm t_0\})$ pointing out of $G\cap \big(M\times (-t_0, t_0)\big)$ and $dA_{\bar{g}}$ is the area element induced by $\bar{g}|_{G\cap (M\times \{\pm t_0\})}$. Translating $G$ vertically as in Proposition \ref{prop: SY Jang equation} (2), $G$ has infinity ends that are $C^{2,\al}$-asymptotic to $(\pa U \times \R)$. Therefore, $G\cap (M \times \{\pm t_0\})$ converges uniformly to $\pa U$ as $t_0\to +\infty$. Then the trace theorem implies that there exists constants $C$, $T>0$ depending only on geometry of $\pa U$ such that for all $t_0> T$
\begin{align*}
    \int_{G\cap \big(M\times \{\pm t_0\}\big)} \phi^2 dA_{\bar{g}} \leq C \int_{G\cap \big(M\times (-t_0, t_0)\big)} |d\phi|_{\bar{g}}^2 + \phi^2 dV_{\bar{g}}.
\end{align*}
Since $\lim_{t_0\to \infty}\eta_\pm = \pm \pa_t$ and $|\bar{\D}\log \beta|\leq c_4$ in Proposition \ref{Prop: LocalEst}, we have
\begin{align*}
    \lim_{t_0 \to \infty}\lan \xi + \bar{\D}\log \beta, \eta_\pm\ran_{\bar{g}} = k(\nu_{\pa U}, \pm\pa_t) \pm \pa_t \log \beta = 0.
\end{align*}
We also perturb $G$ to exact cylinders $\pa U\times \R$ with a new metric $\tilde{g} = g|_{\pa U} + dt^2$ for $t_0-1\leq |t|\leq t_0$ and keep $\tilde{g} = \bar{g}$ for $|t|\leq t_0-2$. By choosing $t_0>0$ large enough, the error term due to perturbation and the boundary integral in (\ref{YP: ineq 1}) are bounded by $\vare$ times $W^{1,2}$-norm of $\phi$ on $G\cap \big(M\times (-t_0, t_0)\big)$ for a very small $\vare>0$. By using the strong dominant energy condition $\mu - |J|_g \geq \de$, the inequality (\ref{YP: ineq 1}) implies
\begin{align}\label{ineq: stability of graph}
    \de \int_{G\cap \big(M\times (-t_0, t_0)\big)} \phi^2 dV_{\tilde{g}}\leq \int_{G\cap \big(M\times (-t_0, t_0)\big)} 3|d\phi|_{\tilde{g}}^2 +\rR_{\tilde{g}}\phi^2dV_{\tilde{g}}.
\end{align}

Let $\Si_i\subset \pa U$ be a connected component and let $\ga^{(i)} = g\big|_{\Si_i}$. By Proposition \ref{prop: stability for MOTS}, we know that $\Si_i$ is a closed stable apparent horizon. Following the same computation for (\ref{YP: ineq 1}) without the presence of boundary integral (since $\Si_i$ is closed), we have for any $\xi\in C^1(\Si_i)$, 
\begin{align}\label{ineq: stability of cylinder}
    \de \int_{\Si_i} \xi^2 dA_{\ga^{(i)}} \leq \int_{\Si_i} (\mu -J(\nu))\xi^2dA_{\ga^{(i)}} \leq \int_{\Si_i} |d\xi|_{\ga^{(i)}}^2 dA_{\ga^{(i)}} + \ka_{\ga^{(i)}}\xi^2 dA_{\ga^{(i)}}.
\end{align}
It follows that the first eigenvalue of 2-dimensional conformal Laplacian on $(\Si_i,\ga^{(i)})$ is positive. Taking $\xi \equiv 1$, we find
\begin{align*}
    0< \int_{\Si_i} \ka_{\ga^{(i)}} dV_{\ga^{(i)}}.
\end{align*}
By Gauss-Bonnet theorem, $\Si_i$ is homeomorphic to $S^2$.

Next we will fill up the opening of $G\cap \big(M\times (-t_0, t_0)\big)$ by gluing a 3-ball to obtain a closed manifold homeomorphic to $U\cup \{P_1, \ldots, P_l\}$ using the trick of path of conformal metrics in \cite{MS}. Recall that each $\Si_i$ is homeomorphic to $S^2$. By abuse of notation, we will identify $\Si_i$ as $S^2$ equipped with metric $\ga^{(i)}$ in the following discussion. By uniformization theorem, there exists $w_i\in C^\infty(S^2)$ such that $\ga^{(i)} = e^{2w_i}\ga_*$ where $\ga_*$ is the standard round metric on $S^2$. Let $\eta(s)$ and $a(s)$ be smooth functions on $(0,3)$ such that $0\leq\eta(s)\leq 1$ for all $s\in (0,3)$,
\begin{align*}
    \eta(s) = \left\{
    \begin{array}{ccc}
        0 & , &\mbox{if $s\in (0,1]$}, \\
        1 & , &\mbox{if $s\in [2,3)$},
    \end{array}\right.
\end{align*}
and $a(s)\leq 0$ for all $s\in (0,3)$,
\begin{align*}
    a(s) = \left\{
    \begin{array}{ccc}
        \log s & , &\mbox{if $s\in (0,\frac{1}{2}]$}, \\
        0 & , &\mbox{if $s\in [2,3)$}.
    \end{array}\right.
\end{align*}
Set $\ga_s^{(i)}(x) = e^{2\eta(s)w_i(x) + 2a(s)}\ga_*(x)$ so that $\ga_s^{(i)} = \ga^{(i)}$ for $s\in [2, 3)$. Then the cylinder $\cC_i:= S^2\times (0,3)$ equipped with the warped product $\ga_s^{(i)} + ds^2$ coincides with flat punctured 3-ball in spherical coordinates for $s\in (0,\frac{1}{2})$, and coincides with $(G\cap \big(\Si_i\times (-t_0, t_0)\big), \tilde{g})$ for $s\in (2,3)$ with the orientation $\pa_s$ pointing into $G\cap \big(M\times (-t_0, t_0)\big)$. In such a way, we can patch up the opening by gluing a 3-ball $\cC_i\cup {P_i}$ where $P_i$ is the origin in spherical coordinates. We repeat the surgery at other cylindrical ends and then we obtain a new smooth closed manifold $(\hat{M}, \hat{g})$ which is homeomorphic to $G\cup\{P_1, \ldots, P_l\}\cong U\cup\{P_1, \ldots, P_l\}$.

To complete the proof, we need to show that $(\hat{M}, \hat{g})$ is of positive Yamabe type. It suffices to show that $\la_1(-\De_{\hat{g}}+\frac{1}{8}\rR_{\hat{g}})$ is positive, since this implies that there exists a smooth positive eigenfunction $u$ of $-\De_{\hat{g}}+\frac{1}{8}\rR_{\hat{g}}$ on $\hat{M}$ such that 
\begin{align*}
    \rR_{u^4\hat{g}} = 8u^{-5}(-\De_{\hat{g}}u+\frac{1}{8}\rR_{\hat{g}} u) = 8u^{-4}\la_1(-\De_{\hat{g}}+\frac{1}{8}\rR_{\hat{g}})>0,
\end{align*}
and therefore $\hat{M}$ admits a metric $u^4\hat{g}$ with positive scalar curvature.
Let $\phi\in C^1(\hat{M})$. The relevant bilinear form can be split as the sum of integrals on several portions
\begin{align*}
    \int_{\hat{M}} |d\phi|_{\hat{g}}^2 + \frac{1}{8} \rR_{\hat{g}}\phi^2 dV_{\hat{g}} 
    = \sum_{i=1}^\ell\int_{\cC_i\cup P_i} + \int_{\hat{M}\bsls \bigcup_j(\cC_j\cup P_j)} |d\phi|_{\hat{g}}^2 + \frac{1}{8} \rR_{\hat{g}}\phi^2 dV_{\hat{g}}.
\end{align*}
For each integral on $\cC_i\cup P_i$, we will use Lemma \ref{lem: gluing} to get a positive lower bound. By definition of $\ga_s^{(i)}$, it is clear that the conditions for Lemma \ref{lem: gluing}  
$\frac{d}{ds}\big|_{s=0^+} e^{2(\eta w_i + a)} = \frac{d}{ds}\big|_{s=3^-} e^{2(\eta w_i + a)} = 0$ and $\sup_{C_i} \big|\frac{d^2}{ds^2}e^{2(\eta w_i + a)}\big|< \infty$ hold true. Recall for any $\varphi\in C^\infty(S^2)$ the Gaussian curvature of conformal metric $e^{2\varphi}\ga_*$ on $S^2$ is given by
\begin{equation}\label{conformal Gaussian curvature}
	\ka_{e^{2\varphi}\ga_*} = e^{-\varphi}\Big(\ka_{\ga_*}- \De_{\ga_*} \varphi \Big).
\end{equation}
Let $\xi\in C^\infty(S^2)$. Using (\ref{conformal Gaussian curvature}), for $s\in (0,3)$ the bilinear form related to 2-dimensional conformal Laplacian on $(S^2, \ga_s^{(i)})$ can be rewritten as
\begin{align*}
    &\quad  \int_{S^2} |d\xi|_{\ga_s^{(i)}}^2 + \ka_{\ga_s^{(i)}}\xi^2 dA_{\ga_s^{(i)}}\\
    &= \int_{S^2} |d\xi|_{\ga_*}^2 + \big\{\ka_{\ga_*}- \De_{\ga_*} \big(\eta(s) w_i(x)+ a(s)\big)\big\}\xi^2 dA_{\ga_*}\\
    &= \int_{S^2} |d\xi|_{\ga_*}^2 + \big(1 - \eta(s) \De_{\ga_*}  w_i\big)\xi^2 dA_{\ga_*}\\
    &= \eta(s)\int_{S^2}  \Big\{ |d\xi|_{\ga_*}^2 + \big(1 - \De_{\ga_*} w_i(x)\big)\xi^2 \Big\} dA_{\ga_*} + (1-\eta(s)) \int_{S^2} \Big\{ |d\xi|_{\ga_*}^2 +  \xi^2\Big\} dA_{\ga_*}\\
    &=: \mathrm{I} + \mathrm{II}.
\end{align*}
To estimate I, we use (\ref{conformal Gaussian curvature}) and (\ref{ineq: stability of cylinder}) to obtain
\begin{align*}
	\mathrm{I} &= \eta(s)\int_{S^2}  \Big\{ |d\xi|_{\ga_*}^2 + \big(1 - \De_{\ga_*} w_i(x)\big)\xi^2 \Big\} dA_{\ga^{(i)}}\\
    	&= \eta(s)\int_{S^2}  \Big\{ |d\xi|_{\ga^{(i)}}^2 + \ka_{\ga^{(i)}}\xi^2 \Big\} dA_{\ga^{(i)}} \\
    &\geq \eta(s)\de\int_{S^2} \xi^2 dA_{\ga^{(i)}}\\
    &\geq \eta(s)\de \inf_{\cC_i} e^{2(1-\eta)w_i - 2a} \int_{S^2} \xi^2 dA_{\ga_s^{(i)}}.
\end{align*}
To estimate II, we use the fact that $\la_1(-\De_{\ga_*})=2$ to obtain
\begin{align*}
	&\quad (1-\eta(s)) \int_{S^2} \Big\{ |d\xi|_{\ga_*}^2 +  \xi^2\Big\} dA_{\ga_*}\\
    	&\geq 3(1-\eta(s))\int_{S^2} \xi^2 dA_{\ga_*}\\
    &\geq  3(1-\eta(s))\inf_{\cC_i} e^{-2\eta w_i - 2a} \int_{S^2} \xi^2 dA_{\ga_s^{(i)}}.
\end{align*}
Then we can conclude that
\begin{align*}
	\int_{S^2} |d\xi|_{\ga_s^{(i)}}^2 + \ka_{\ga_s^{(i)}}\xi^2 dA_{\ga_s^{(i)}}\geq \big( \eta(s)\de \inf_{\cC_i} e^{2(1-\eta)w_i - 2a} + 3(1-\eta)\inf_{\cC_i} e^{-2\eta w_i - 2a}\big) \int_{S^2} \xi^2 dA_{\ga_s^{(i)}}.
\end{align*}
Since $a\leq 0$ and $0\leq \eta\leq 1$, the coefficient of the integral on the right is positive for all $s\in (0, 3)$. It follows that there exists $\la_*>0$ such that $\la_1(-\De_{\ga_s^{(i)}} + \ka_{\ga_s^{(i)}})\geq \la_*$ for all $s\in (0,3)$. We repeat the argument on all $\cC_i\cup P_i$'s and we may assume $\la_*$ is a lower bound of $\la_1(-\De_{\ga_s^{(i)}} + \ka_{\ga_s^{(i)}})$ for all $\cC_i\cup P_i$'s. Using Lemma \ref{lem: gluing}, we conclude that
\begin{align*}
    \sum_{i=1}^\ell\int_{\cC_i\cup P_i} |d\phi|_{\hat{g}}^2 + \frac{1}{8} \rR_{\hat{g}}\phi^2 dV_{\hat{g}} \geq \frac{\la_*}{4} \sum_{i=1}^\ell \int_{\cC_i\cup P_i} \phi^2 dV_{\hat{g}}.
\end{align*}
From (\ref{ineq: stability of graph}), we find
\begin{align*}
    \int_{\hat{M}\bsls \bigcup_j(\cC_j\cup P_j)} |d\phi|_{\hat{g}}^2 + \frac{1}{8} \rR_{\hat{g}}\phi^2 dV_{\hat{g}} \geq \frac{1}{8}\int_{\hat{M}\bsls \bigcup_j(\cC_j\cup P_j)} 3|d\phi|_{\hat{g}}^2 +  \rR_{\hat{g}}\phi^2 dV_{\hat{g}}\geq \frac{1}{8}\de \int_{\hat{M}\bsls \bigcup_j(\cC_j\cup P_j)} \phi^2 dV_{\hat{g}}.
\end{align*}
Putting all together, we conclude that there exists $\al = \al(\la_*, \de)>0$ such that
\begin{align*}
    \int_{\hat{M}} |d\phi|_{\hat{g}}^2 + \frac{1}{8} \rR_{\hat{g}}\phi^2 dV_{\hat{g}} \geq \al\int_{\hat{M}} \phi^2 dV_{\hat{g}}.
\end{align*}
\end{proof}

The following remarkable theorem classifies the topology of connected, orientable, closed, Yamabe-positive 3-manifolds.
\begin{prop}[Gromov-Lawson, Schoen-Yau cf.\cite{C} Theorem 2.1]\label{prop: classification of Yamabe positive 3-manifold}
Let $X^3$ be a connected, orientable, compact manifold without boundary with positive Yamabe type. Then $X$ is homeomorphic to a connected sum of finite number of spherical space forms $S^3/\Ga$, where $\Ga$ is a finite subgroup of $SO(4)$ acting freely on $S^3$, and $S^2\times S^1$.
\end{prop}

Now we are ready to combine Proposition \ref{prop: classification of Yamabe positive 3-manifold} for the special case, the Classification Theorem \ref{prop: classification of Yamabe positive 3-manifold} together with the Structure Theorem \ref{thm: structure theorem} of blowup regions to prove Theorem \ref{thm: topological condition}.
\begin{proof}[Proof of Theorem \ref{thm: topological condition}.] Without the assumption that there are only finitely many closed smooth marginally stable CES in compact sets in $(M, g, k)$, the Structure Theorem \ref{thm: structure theorem} implies that
\begin{align*}
    \bar{\Om} = (\bigcup_{m = 1}^{N_1} U_m ) \cup (\bigcup_{n = 1}^{N_2} \Phi_n([0, b_n]\times \Si_n))
\end{align*}
where $1\leq N_1, N_2 \leq \infty$, $U_m$ is a maximal domain of solution to Jang equation for all $m$ and $\Phi_n$ is a smooth foliation of closed MOTS or MITS for all $n$. There may be infinitely many maximal domains $U_m$'s. But since the blowup region $\Om$ is bounded, all except finitely many $U_m$'s are \emph{thin} as defined in Proposition \ref{prop: thin domain}. Proposition \ref{prop: thin domain} implies that each \emph{thin} $U_m$ is homeomorphic to a cylinder over its boundary component and Proposition \ref{prop: maximal domain is Yamabe positive} implies that the boundary components of \emph{thin} $U_m$ are 2-spheres. Therefore, all \emph{thin} $U_m$'s are homeomorphic to round cylinder $S^2 \times \R$ and contribute nothing to the topological structure of entire connected sum. 

Every boundary component of $\Phi_n([0, b_n]\times \Si_n)$ is a connected component of $\pa U_m$ or $\pa \Om$ which is a 2-sphere by Proposition \ref{prop: maximal domain is Yamabe positive} and Remark \ref{rmk: boundary is 2-sphere}. Thus, each foliation $\Phi_n([0, b_n]\times \Si_n)$ is homeomorphic to a round cylinder $[0, b_n]\times S^2$ (which may degenerate to $\{0\}\times S^2$).

The main contributions to the topological structure of blowup region come from finitely many \emph{thick} maximal domains. Combining Proposition \ref{prop: maximal domain is Yamabe positive} and Proposition \ref{prop: classification of Yamabe positive 3-manifold}, the compactification of each \emph{thick} maximal domain $U_m$ by adding a point to each boundary component is homeomorhpic to a connected sum of finite number of spherical space forms $S^3/\Ga$ and $S^2\times S^1$. On the other hand, we may view thin maximal domains and foliations as cylindrical necks connecting finitely many \emph{thick} maximal domains in the entire connected sum. Consequently, the compactification $\Om\cup\{P_1, \ldots, P_l\}$ is homeomorphic to a connected sum of finite number of spherical space forms $S^3/\Ga$ and $S^2\times S^1$.
\end{proof}

\begin{cor}
Suppose the dominant energy condition holds strictly, i.e. $\mu > |J|$. Let $u$ be a capillary blowdown limit of $f_s$ and $\Om\subset \Om_+$ with boundary components $\Si_1, \ldots, \Si_l$. If the compactification $\Om\cup \{P_1,\ldots, P_l\}$ by adding a point to each boundary component is not homeophorphic to a connected sum of finite number of spherical space forms $S^3/G$ and $S^2\times S^1$, then $u$ is not trivial in $\Om$.
\end{cor}

\end{document}